\documentclass[11pt]{amsart}
\usepackage{fullpage,verbatim,amssymb,mathtools}
\usepackage{hyperref,subfig}
\usepackage[active]{srcltx}
\usepackage[usenames,dvipsnames]{color}
\usepackage[normalem]{ulem}

\makeatletter
\let\@@pmod\mod
\DeclareRobustCommand{\mod}{\@ifstar\@pmods\@@pmod}
\def\@pmods#1{\mkern4mu({\operator@font mod}\mkern 6mu#1)}
\makeatother

\definecolor{blue}{rgb}{0,0,1}
\definecolor{red}{rgb}{1,0,0}
\definecolor{green}{rgb}{0,.6,.2}
\definecolor{purple}{rgb}{1,0,1}

\long\def\red#1\endred{\textcolor{red}{#1}}
\long\def\blue#1\endblue{\textcolor{blue}{#1}}
\long\def\purple#1\endpurple{\textcolor{purple}{ #1}}
\long\def\green#1\endgreen{\textcolor{green}{#1}}

\newcommand{\sm}{\left(\begin{smallmatrix}}
\newcommand{\esm}{\end{smallmatrix}\right)}
\newcommand{\bpm}{\begin{pmatrix}}
\newcommand{\ebpm}{\end{pmatrix}}

\newcommand{\Z}{\mathbb{Z}}
\newcommand{\R}{\mathbb{R}}
\newcommand{\Q}{\mathbb{Q}}
\newcommand{\C}{\mathbb{C}}
\newcommand{\N}{\mathcal{N}}
\newcommand{\M}{\mathfrak{M}}
\newcommand{\m}{\mathfrak{m}}
\renewcommand{\SS}{\Sigma}

\DeclareMathOperator{\Tr}{Tr}
\DeclareMathOperator{\ord}{ord}
\DeclareMathOperator{\Res}{Res}

\newtheorem{theorem}{Theorem}
\newtheorem{lemma}[theorem]{Lemma}
\newtheorem{proposition}[theorem]{Proposition}

\theoremstyle{remark}
\newtheorem*{remarks}{Remarks}
\numberwithin{theorem}{section}
\numberwithin{equation}{section}

\begin{document}
\title{Murmurations of modular forms in the weight aspect}
\author{Jonathan Bober}
\author{Andrew R. Booker}
\author{Min Lee}
\author{David Lowry-Duda}

\thanks{The first and second authors are supported by the Heilbronn Institute for Mathematical Research.
The third author is supported by a Royal Society University Research Fellowship.
The fourth author is supported by the Simons Collaboration in Arithmetic Geometry, Number Theory, and Computation via the Simons Foundation grant 546235.}

\begin{abstract}
We prove the existence of ``murmurations'' in the family of holomorphic modular forms of level $1$ and weight $k\to\infty$, that is, correlations between their root numbers and Hecke eigenvalues at primes growing in proportion to the analytic conductor. This is the first demonstration of murmurations in an archimedean family.
\end{abstract}
\maketitle

\section{Introduction}
Using machine learning algorithms, He, Lee, Oliver, and Pozdnyakov \cite{HLOP} discovered an apparent correlation between the root numbers of elliptic curve $L$-functions and their Dirichlet coefficients $a_p$ at primes $p$ varying in proportion to the conductor. They dubbed this correlation ``murmurations of elliptic curves'' due to the resemblance of their graphs to the swarming patterns of flocks of birds (see Figure~\ref{fig:murmurations}).
Further experimental work by the authors of \cite{HLOP} and Sutherland \cite{sutherland,HLOPS} has shown this to be a general phenomenon, present in many natural families of $L$-functions.

\begin{figure}[ht!]
\includegraphics[width=0.64\textwidth]{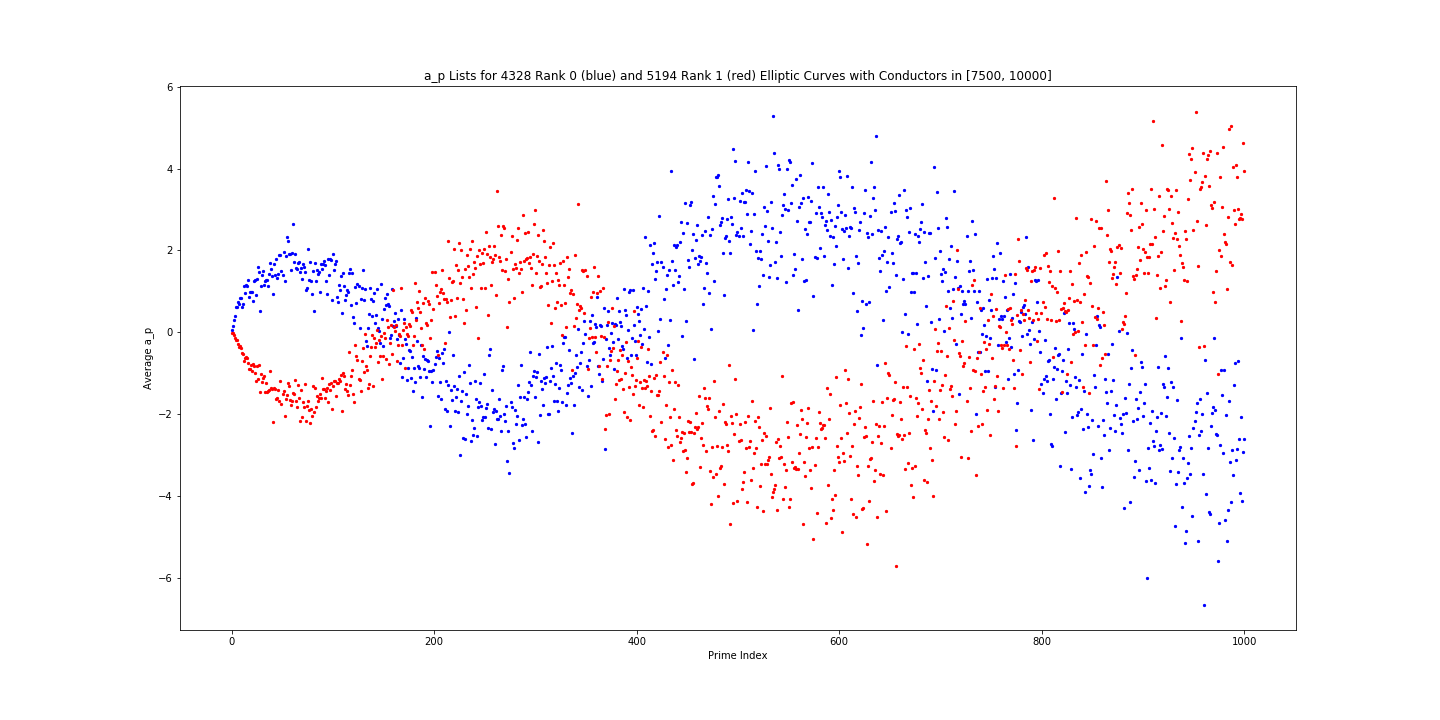}

\;\;\,\includegraphics[width=0.49664\textwidth]{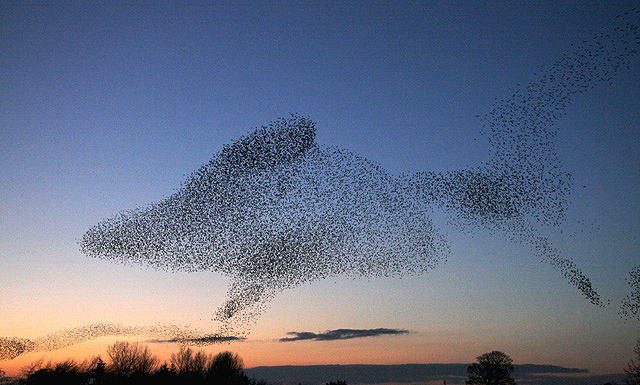}
\caption{top: Average of $a_p$ over isogeny classes of elliptic curves of conductor in $[7500,10000]$ and fixed rank (\textcolor{blue}{blue} = rank 0, \textcolor{red}{red} = rank 1), for primes $p\in[2,7919]$\\(reprinted from \cite{HLOP} with permission from the authors).\\bottom: \emph{Starling shapes in the evening sky}.\\cc-by-sa/2.0 -- \textcopyright Walter Baxter -- \href{https://www.geograph.org.uk/photo/1065181}{geograph.org.uk/p/1065181}}\label{fig:murmurations}
\end{figure}

Zubrilina \cite{zubrilina} has provided the first theoretical confirmation, proving the existence of murmurations in the family of self-dual holomorphic newforms of a fixed weight $k$ and squarefree conductor $N\to\infty$. More precisely, for $k\in2\Z_{>0}$ and $N\in\Z_{>0}$, let $H_k(N)$ be a basis of normalized Hecke eigenforms for $S_k^{\mathrm{new}}(\Gamma_0(N))$. For each $f\in H_k(N)$, let $\epsilon_f$ be its root number, and let $\lambda_f(n)$ be its normalized Hecke eigenvalues, so that
\[
f(z)=\sum_{n=1}^\infty\lambda_f(n)n^{\frac{k-1}{2}}e^{2\pi inz}.
\]
Then Zubrilina showed that there exists a continuous function $M_k:\R_{>0}\to\R$ such that, for fixed $y\in\R_{>0}$ and $\delta\in(0,1)$, we have
\[
\lim_{\substack{p\,\mathrm{prime}\\p\to\infty}}
\frac{\sum_{\substack{N\in[p/y,p/y+p^\delta]\cap\Z\\N\,\mathrm{squarefree}}}\sum_{f\in H_k(N)}\epsilon_f\lambda_f(p)\sqrt{p}}{\sum_{\substack{N\in[p/y,p/y+p^\delta]\cap\Z\\N\,\mathrm{squarefree}}}\sum_{f\in H_k(N)}1}=M_k(y).
\]
A comparison of $M_2(y)$ with numerical data is shown in Figure~\ref{fig:M_2}.

\begin{figure}
\includegraphics[width=\textwidth]{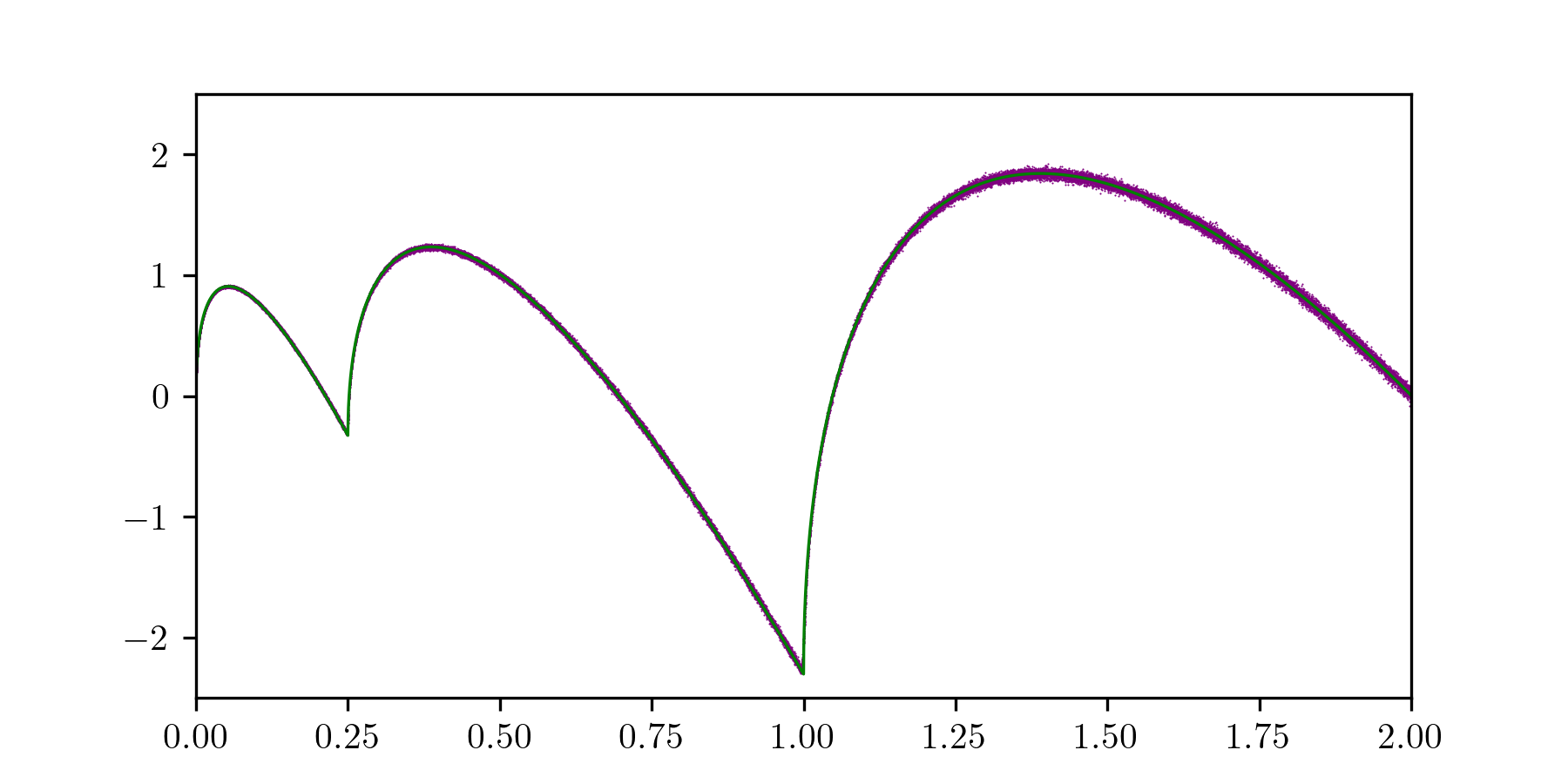}
\caption{A comparison of $M_2(y)$ for $y\in[0,2]$ (green) and the points $\bigl(p/2^{18},r(p)\bigr)$ for primes $p<2^{19}$ (purple), where
$r(p)=\frac{\sum_{N\in I}\sum_{f\in H_2(N)}\epsilon_f\lambda_f(p)\sqrt{p}}
        {\sum_{N\in I}\sum_{f\in H_2(N)}1}$, and $I$ is the set of all squarefree integers in the range $[2^{18} \pm 2^{10}]$ (data computed by Andrew Sutherland).}\label{fig:M_2}

\end{figure}
In this paper, following a suggestion of Sarnak, we investigate the murmuration phenomenon in a family of $L$-functions with varying archimedean parameters, namely the modular forms of level $1$ and weight $k\to\infty$. In this family, the root number $\epsilon_f$ is simply $(-1)^{\frac{k}{2}}$, so we expect to see biases in the $\lambda_f(p)$ values depending on the residue class of $k\bmod4$.
Following Rubinstein (unpublished, but see \cite{andy-distributions} and \cite{LMFDB}), we define the \emph{analytic conductor}
of $f\in H_k(1)$ to be the positive real number
\[
\N(k) := \left(\frac{\exp\psi(k/2)}{2\pi}\right)^2,
\]
where $\psi(x) = \frac{\Gamma'(x)}{\Gamma(x)}$; by Stirling's formula, $\N(k)=\left(\frac{k-1}{4\pi}\right)^2+O(1)$.
In analogy with the arithmetic conductor scaling in \cite{HLOP,HLOPS,zubrilina},
one might expect to see murmurations for $p$ growing in proportion to $\N(k)$; our main result confirms this expectation under the Generalized Riemann Hypothesis (GRH):
\begin{theorem}\label{thm:main}
Assume GRH for the $L$-functions of Dirichlet characters and modular forms.
Fix $\varepsilon\in(0,\frac1{12})$, $\delta\in\{0,1\}$, and a compact interval $E\subset\R_{>0}$ with $|E|>0$. Let $K,H\in\R_{>0}$ with $K^{\frac56+\varepsilon}<H<K^{1-\varepsilon}$, and set $N=\N(K)$. Then as $K\to\infty$, we have
\begin{equation}\label{e:main_ratio}
\frac{\sum_{\substack{p\,\mathrm{prime}\\p/N\in E}}\log{p}\sum_{\substack{k\equiv2\delta\bmod4\\|k-K|\le H}}\sum_{f\in H_k(1)}\lambda_f(p)}{\sum_{\substack{p\,\mathrm{prime}\\p/N\in E}}\log{p}\sum_{\substack{k\equiv2\delta\bmod4\\|k-K|\le H}}\sum_{f\in H_k(1)}1}
= \frac{(-1)^\delta}{\sqrt{N}}\left(\frac{\nu(E)}{|E|}+
o_{E,\varepsilon}(1)
\right),
\end{equation}
where
\begin{equation}\label{e:nu_def}
\nu(E)=\frac{1}{\zeta(2)}\sideset{}{^\ast}\sum_{\substack{a,q\in\Z_{>0}\\\gcd(a,q)=1\\(a/q)^{-2}\in E}}\frac{\mu(q)^2}{\varphi(q)^2\sigma(q)}\left(\frac{q}{a}\right)^3
=\frac12\sum_{t=-\infty}^\infty\prod_{p\nmid t}\frac{p^2-p-1}{p^2-p}\cdot \int_E\cos\!\left(\frac{2\pi t}{\sqrt{y}}\right)dy,
\end{equation}
and the $\ast$ indicates that terms occurring at the endpoints of $E$ are halved.
\end{theorem}
\noindent A comparison of $\nu(E)$ with numerical data is shown in Figure~\ref{fig:nu}.

\begin{figure}[hbt!]\label{fig:sum-and-measure}
\includegraphics[width=\textwidth]{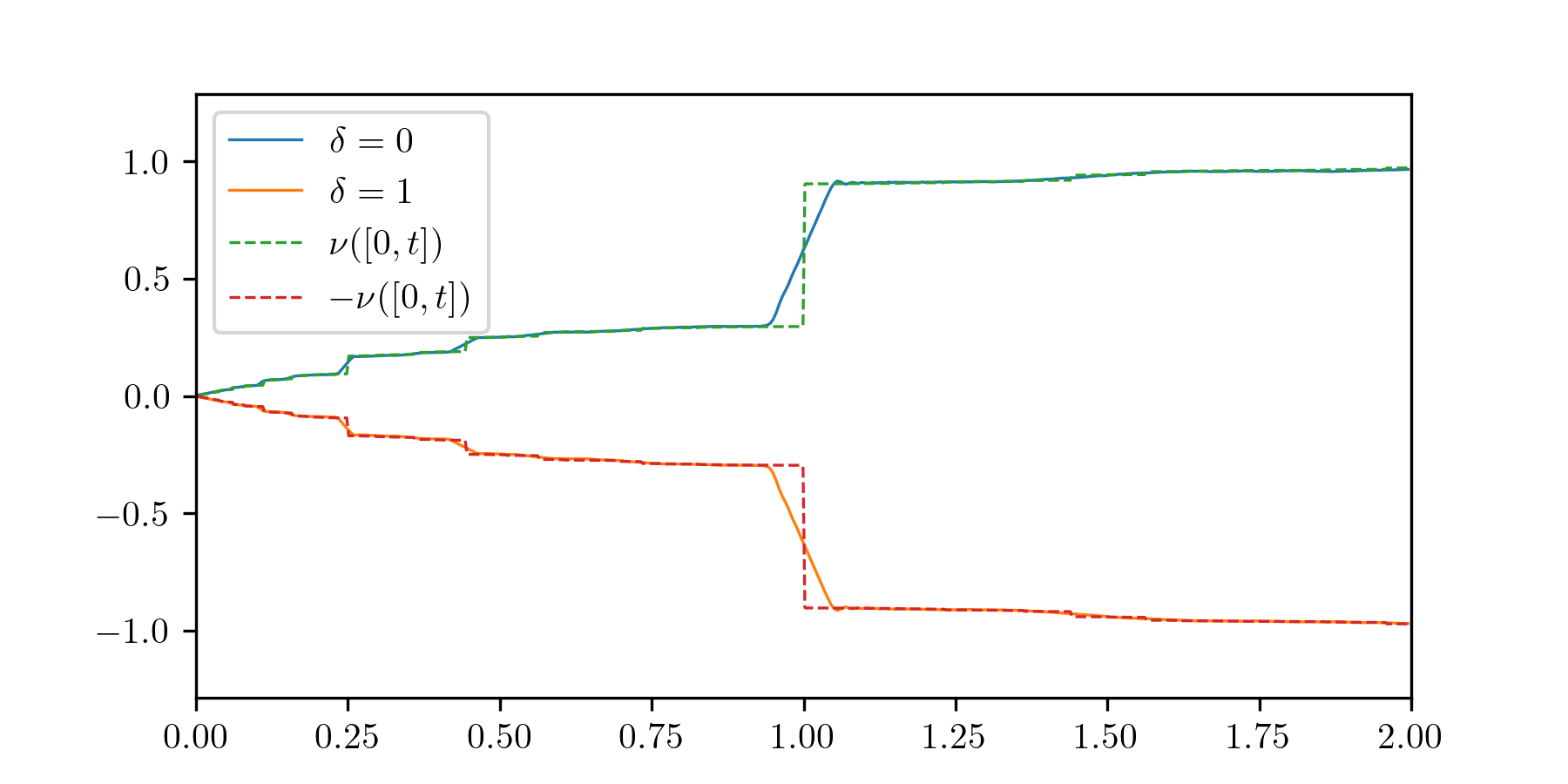}
\caption{A comparison of $(-1)^\delta\nu((0,t])$ and the left-hand side of \eqref{e:main_ratio} scaled by $t\sqrt{N}$, for $K=3850$, $H=100$, and $t \in [0,2].$}\label{fig:nu}
\end{figure}

\clearpage

\begin{remarks}\
\begin{enumerate}
\item We have averaged the normalized $\lambda_f(p)$ values directly, which shows a correlation of size $1/\sqrt{N}$ with the root number. As in \cite{HLOP,HLOPS,zubrilina,sarnak}, one could instead average $\lambda_f(p)\sqrt{p}$ to boost the correlation. A similar result holds for that average, with the right-hand side becoming $(-1)^\delta\bigl(\frac{\nu^\sharp(E)}{|E|}+o_{E,\varepsilon}(1)\bigr)$, where $d\nu^\sharp(y)=\sqrt{y}\,d\nu(y)$, i.e.\
\[
\nu^\sharp(E)=\frac{1}{\zeta(2)}\sideset{}{^\ast}\sum_{\substack{a,q\in\Z_{>0}\\\gcd(a,q)=1\\(a/q)^{-2}\in E}}\frac{\mu(q)^2}{\varphi(q)^2\sigma(q)}\left(\frac{q}{a}\right)^4
=\frac12\sum_{t=-\infty}^\infty\prod_{p\nmid t}\frac{p^2-p-1}{p^2-p}\cdot \int_E\sqrt{y}\cos\!\left(\frac{2\pi t}{\sqrt{y}}\right)dy. 
\]
\item
Zubrilina's results exhibit correlations for an individual prime $p$, requiring only a short average over conductors. One can carry out some of the analysis for a fixed $p$ with our family, but the resulting correlation depends on the residue classes of $p$ modulo small primes; at least a short average over $p$ is needed to get a universal result depending only on the size of $p/N$. Sarnak \cite{sarnak} has speculated that this is closely related to the conductor growth of the family, and that an average over $p$ is necessary to see murmurations in any family with at most $O(x)$ $L$-functions of analytic conductor $\le x$. (Note that our family contains $\asymp x$ forms of analytic conductor $\le x$, compared to $\asymp_k x^2$ in Zubrilina's family.)
\item
Our result depends on GRH in order to compute the sums over primes with power-saving accuracy. However, GRH for modular form $L$-functions is only needed to handle the sharp cutoff in the sum over $k$ in \eqref{e:main_ratio}, and could be dispensed with in a smoothed version. The combination of a weak signal (only $1/\sqrt{N}$ in size) and a distribution that is not absolutely continuous makes the proof delicate, requiring many nontrivial estimates throughout, including for this truncation error. (The large sieve does not suffice here, since the number of forms involved is small relative to the size of $p$.)
\item
Closely related to murmurations is the $1$-level density, which involves computing $a_p$ averages in a family of $L$-functions for $p$ on the scale of $N^\theta$ for some real number $\theta>0$. For our family of holomorphic forms in the weight aspect, the $1$-level density statistics for $\theta<2$ were computed by Iwaniec, Luo, and Sarnak \cite{ILS}, demonstrating a sharp phase transition at $\theta=1$, in accordance with predictions from random matrix theory \cite{KS}. Sarnak \cite{sarnak} has pointed out that, in retrospect, one can see murmurations in \cite{ILS} by examining the $\theta=1$ transition more closely. One technical difference with our work is that \cite{ILS} used the Petersson trace formula (rather than Eichler--Selberg), which sums the $a_p$ values of an $L^2$-normalized basis of forms. This changes the distribution, but it shares some of the same qualitative features, including point masses at squares of squarefree integers (corresponding to the terms of \eqref{e:nu_def} with $a=1$).
\item\label{remark:interval}
The $o(1)$ error term in Theorem~\ref{thm:main} depends on Diophantine properties of the endpoints of $E$. More precisely, if we write $E=[u,v]$, then the theorem holds with $o_{E,\varepsilon}(1)$ replaced by
\[
O_{E,\varepsilon}\!\left(\left(\frac{H}{K^{1-\varepsilon}}\right)^{\frac19}+\left(\frac{H}{K^{\frac56+\varepsilon}}\right)^{-\frac{6}{11}}\right),
\]
provided that $\sqrt{u}$ and $\sqrt{v}$ have irrationality measures at most $27$. Generally finite irrationality measure implies a power-saving error term, but for Liouville numbers the convergence can be arbitrarily slow.
\item Our proof should carry over to the family of Maass forms of level $1$ and Laplace eigenvalue $\lambda\to\infty$ with appropriate modifications. The main difference is that one would encounter discriminants $t^2+4n$ in the trace formula instead of $t^2-4n$ (see \eqref{e:STF_nprime}), but the local analysis in Lemmas~\ref{lem:psiD_modm^2}--\ref{lem:L1_def} is insensitive to this change.
\item\label{rem:discrepancy}
Sutherland \cite{sutherland-talk} has empirically observed that the signal is less noisy if one computes correlation sums of $\epsilon a_p$ rather than separating by root number, pointing to the existence of a lower-order discrepancy. Our proof suggests a possible source for this in our family, namely the terms of \eqref{e:Tkn_smooth_k} with $\ell=\pm1$, which vanish when considering the correlation. As we show in Lemma~\ref{lem:Poisson_error}, those terms have lower order than the main term, and they are eventually swamped by other error terms in our analysis, so it may not be possible to isolate the discrepancy theoretically.
\item
Our distribution $\nu$ shows many qualitative differences from Zubrilina's $M_k$; in particular, it has infinitely many point masses, and no sign changes---the terms of \eqref{e:nu_def} are all nonnegative. It is conceivable that $M_2$ and $\nu^\sharp$ are the margins of some joint distribution in the large analytic conductor limit ($Nk^2\to\infty$). It would be interesting to try to merge the two results by taking a double limit; in particular, do $\lim_{N\to\infty}\lim_{k\to\infty}$ and $\lim_{k\to\infty}\lim_{N\to\infty}$ both exist, and are they equal?
\item
Although murmurations have hitherto been investigated as a pattern in $a_p$ values at \emph{primes}, \eqref{e:main_ratio} would continue to make sense with the sums over $p$ replaced by sums over integers $n$. The analysis becomes much simpler with sums over integers, and there is still a discernible murmuration: the analogue of $\nu(E)$ is
\[
\sideset{}{^\ast}\sum_{\substack{a\in\Z_{>0}\\a^{-2}\in E}}a^{-3}.
\]
It may be of interest to look for integer murmurations in other families.
\end{enumerate}
\end{remarks}

\subsection*{Acknowledgements}
We thank Peter Sarnak and Nina Zubrilina for insightful comments, and Andrew Sutherland for helping to popularize this topic and sharing his data.
This project began at the workshop ``Murmurations in Arithmetic'' and the conference ``LMFDB, Computation, and Number Theory'', both held at ICERM in July 2023.
We thank ICERM and the organizers for their hospitality.

\section{The trace formula and overview of the proof}
The main tool that we use in the proof of Theorem~\ref{thm:main} is the Eichler--Selberg trace formula, which allows us to evaluate the sums over $f\in H_k(1)$ in \eqref{e:main_ratio}.
For a positive integer $n$, let $T_n$ denote the $n$th Hecke operator acting on $S_k(1)$, and let
\[
\Tr T_n|S_k(1)=\sum_{f\in H_k(1)}\lambda_f(n)n^{\frac{k-1}{2}}
\]
denote its trace. Then by \cite[Theorem 2.1]{child} (see also \cite{eic, hij, Zagier77, Cohen77, Sv}), we have
\[
\Tr T_n|S_k(1) = A_1-A_2-A_3+A_4, 
\]
where 
\[
A_1 = \begin{cases}\frac{n^{\frac{k}{2}-1} (k-1)}{12}&\text{if }\sqrt{n}\in\Z,\\
0&\text{otherwise},
\end{cases}
\]
\[
A_2 = \sum_{\substack{t\in\Z,\,t^2-4n=d\ell^2<0\\ \rho^2-t\rho+n=0,\,\Im\rho>0}}
\frac{\rho^{k-1}-\bar{\rho}^{k-1}}{\rho-\bar{\rho}} 
\frac{h(d)}{w(d)} \prod_{p\mid \ell} \bigg(p^{{\rm ord}_p(\ell)}+\bigg(1-\left(\frac{d}{p}\right)\bigg) \frac{p^{{\rm ord}_p(\ell)}-1}{p-1}\bigg), 
\] 
\[
A_3 = \frac12\sum_{d\mid n}\min(d,n/d)^{k-1},
\quad\text{and}\quad
A_4 = \begin{cases}
\sigma(n)&\text{if }k=2,\\
0&\text{otherwise}.
\end{cases}
\]
Here $d$ is the discriminant of $\Q\bigl(\sqrt{t^2-4n}\bigr)$, $h(d)$ is the class number, and $w(d)$ is the number of units in the ring of integers.

Following \cite[\S1.1]{BL}, for a discriminant $D=d\ell^2$ we define $\psi_D(m)=\left(\frac{d}{m/\gcd(m,\ell)}\right)$, with the convention that $\psi_0(m)=1$. Then, by Dirichlet's class number formula, for $D<0$ we have
\[
L(1,\psi_D):=\sum_{m=1}^\infty\frac{\psi_D(m)}{m}
=\frac{2\pi h(d)}{w(d)\sqrt{|D|}}
\prod_{p\mid \ell} \bigg(p^{{\rm ord}_p}(\ell) + \bigg(1-\left(\frac{d}{p}\right)\bigg) \frac{p^{{\rm ord}_p(\ell)}-1}{p-1}\bigg).
\]
Thus $A_2$ can be written in the form 
\[
A_2  
= \sum_{\substack{t\in\Z,\,t^2<4n\\ \rho^2-t\rho+n=0,\,\Im\rho>0}}
\frac{\rho^{k-1}-\bar{\rho}^{k-1}}{\rho-\bar{\rho}} 
\frac{\sqrt{4n-t^2}}{2\pi}L(1,\psi_{t^2-4n}).
\]

Next, for $t\in\Z$ with $t^2<4n$, define
\[
\phi_{t,n}=\arcsin\!\left(\frac{t}{2\sqrt{n}}\right)
\in\left(-\frac\pi2,\frac\pi2\right).
\]
Then we have $\rho=\sqrt{n}e^{i(\frac\pi2-\phi_{t,n})}$,
where $\rho$ is the solution to $\rho^2-t\rho+n=0$ with positive imaginary part.
Thus,
\begin{align*}
\frac{\rho^{k-1}-\bar{\rho}^{k-1}}{\rho-\bar{\rho}} &= n^{\frac{k}{2}-1} \frac{\sin((k-1)(\frac\pi2-\phi_{t,n}))}{\sin(\frac\pi2-\phi_{t,n})}\\
&=-(-1)^{\frac{k}{2}}n^{\frac{k}{2}-1}\frac{\cos((k-1)\phi_{t,n})}{\cos(\phi_{t,n})}
=-2(-1)^{\frac{k}{2}}n^{\frac{k-1}{2}}\frac{\cos((k-1)\phi_{t,n})}{\sqrt{4n-t^2}}.
\end{align*}

Assume now that $n$ is prime and $k>2$. Then, combining these expressions, we obtain
\begin{equation}\label{e:STF_nprime}
\sum_{f\in H_k(1)}\lambda_f(n) = n^{\frac{1-k}{2}} \Tr T_n|S_k(1)
= - n^{\frac{1-k}{2}}
+ \frac{(-1)^{\frac{k}{2}}}{\pi}\sum_{\substack{t\in\Z\\t^2<4n}}\cos((k-1)\phi_{t,n})
L(1, \psi_{t^2-4n}).
\end{equation}

\subsection{Overview of the proof of Theorem \ref{thm:main}}
We conclude this section with a sketch of the proof of Theorem~\ref{thm:main}, with the details carried out in Sections~\ref{sec:ksum}--\ref{sec:circle}.

For a fixed choice of $\delta, K, H$, and $E$, we define $\SS=\SS(\delta,K,H,E)$ to be the numerator on the left-hand side of \eqref{e:main_ratio}, viz.\
\begin{equation*}
  \SS
 = \sum_{\substack{n\,\mathrm{prime}\\n/N\in E}} \log{n}
  \sum_{\substack{k\equiv2\delta\bmod4\\\lvert k-K \rvert \le H}}
  \sum_{f\in H_k(1)}\lambda_f(n).
\end{equation*}
Having computed the sums over $f \in H_k(1)$, we see that
\[
    \SS = \sum_{\substack{n\,\mathrm{prime}\\n/N\in E}}\log{n}
            \sum_{\substack{k\equiv2\delta\bmod4\\|k-K|\le H}}
            \Biggl(-n^\frac{1-k}{2} + \frac{(-1)^{\delta}}{\pi}
                \sum_{\substack{t\in\Z\\t^2<4n}}\cos((k-1)\phi_{t,n})
                    L(1, \psi_{t^2-4n})\Biggr).
\]
As $k\to\infty$, the contribution from $n^\frac{1-k}{2}$ is negligible,
so we are left with the sum
\[
    \SS \approx \frac{(-1)^\delta}{\pi}\sum_{\substack{n\,\mathrm{prime}\\n/N\in E}}\log{n}
            \sum_{\substack{k\equiv2\delta\bmod4\\|k-K|\le H}}
            \sum_{\substack{t\in\Z\\t^2<4n}}\cos((k-1)\phi_{t,n})
                L(1, \psi_{t^2-4n}).
\]
We will shortly bring the sum over $k$ inside, but even before applying the
trace formula, we take a partition of unity to split the sum over $k$ into
smooth weighted sums over shorter ranges of length $\asymp h$. 
For a fixed choice of smooth function 
$W : \R \rightarrow \R_{\ge 0}$ supported on $[-1,1]$
and parameters $h, J$ and $K_0$ (depending on $H$ and $K$), we find (assuming GRH for
$L(s,f)$) that
\[
    \SS =  \sum_{\substack{n\,\mathrm{prime}\\n/N\in E}}\log{n}
        \sum_{j=0}^{J-1}\sum_{k \in 2\delta + 4\Z} W\!\left(\frac{k - (K_0 + 4jh)}{4h}\right)
        \sum_{f\in H_k(1)}\lambda_f(n) + O_{E, \varepsilon}\bigl(hK^{2 + \varepsilon}\bigr).
\]
(The denominator on the left-hand side of \eqref{e:main_ratio} is of size $\asymp HK^3$, so after dividing by $\sqrt{N} \asymp K$ we will need an error term that is
$o(HK^2)$ for the claimed theorem; the size of $h$ depends on $H$, but it will be always be at most $H^{1-\epsilon}$.
See the beginning of Section~\ref{sec:ksum}.)

Upon reordering the summation and applying the trace formula, this becomes
\begin{multline*}
    \SS =  \frac{(-1)^\delta}{\pi} \sum_{j=0}^{J-1}
        \sum_{\substack{n\,\mathrm{prime}\\n/N\in E}}\log{n}
        \sum_{\substack{t\in\Z\\t^2<4n}} L(1,\psi_{t^2-4n})
        \sum_{k \in 2\delta + 4\Z} W\!\left(\frac{k - (K_0 + 4jh)}{4h}\right)
            \cos((k-1)\phi_{t,n})
                \\
                + O_{E, \varepsilon}\bigl(hK^{2 + \varepsilon}\bigr).
\end{multline*}
Since $k$ runs over a fixed residue class mod $4$, the inner sum is effectively concentrated around values of $\phi_{t,n}$ which are close to $0$ or $\pm \frac\pi2$. However, as there
are relatively few integers $t$ with $t^2$ close to $4n$, the terms corresponding to
$\phi_{t,n}\approx\pm\frac\pi2$ do not contribute to leading order.
We see this when we apply the Poisson summation formula to the innermost sum. We will have
\[
\sum_{k \in 2\delta + 4\Z} W\!\left(\frac{k - k_0}{4h}\right) \cos((k-1)\phi_{t,n})
    = h \cos((k_0 - 1)\phi_{t,n})\sum_{\ell \in \Z} \widehat W\bigl(h\ell + 2h\phi_{t,n}/\pi\bigr),
\]
and it transpires that only the term $\ell = 0$ is large enough to overcome our error terms.
(See Remark \ref{rem:discrepancy} after Theorem \ref{thm:main} about the terms with $\ell = \pm 1$, however.) Specifically, after Lemma \ref{lem:Poisson_error} we find that
\begin{multline*}
    \SS =  \frac{(-1)^\delta h}{\pi} \sum_{\substack{k_0 \equiv K_0 \bmod{4h} \\ 0 \le \frac{k_0 - K_0}{4h} < J}}
        \sum_{\substack{n\,\mathrm{prime}\\n/N\in E}}\log{n}
        \sum_{\substack{t\in\Z\\t^2<4n}} L(1,\psi_{t^2-4n})
        \cos((k_0 - 1)\phi_{t,n})\widehat{W}\bigl(2h\phi_{t,n}/\pi\bigr)
                \\
                + O_{E, \varepsilon}\!\left(hK^{2 + \varepsilon} + \frac{HK^3 \log K}{h^2}\right).
\end{multline*}
Due to the rapid decay of $\widehat W$ we can now restrict the sum over $t$
to a smaller range and put in a linear approximation for $\phi_{t,n}$, and at
the end of Section \ref{sec:ksum} we arrive at
\begin{multline*}
    \SS =  \frac{(-1)^\delta h}{\pi} \sum_{k_0}
        \sum_{\substack{n\,\mathrm{prime}\\n/N\in E}}\log{n}
        \sum_{\substack{t\in\Z\\|t| \le T}} L(1,\psi_{t^2-4n})
        \cos\!\left(\frac{(k_0 - 1)t}{2\sqrt{n}}\right)\widehat{W}\!\left(\frac{ht}{\pi\sqrt{n}}\right)
                \\
                + O_{E, \varepsilon, \varepsilon_0}\!\left(hK^{2 + \varepsilon} + \frac{HK^3 \log K}{h^2}\right)
\end{multline*}
for $T = K^{1 + \varepsilon_0}/h$. (Note that as our interval does not contain $0$,
we always have $n \gg K^2$, so that this summation over $t$ really is a truncation.)

The next step, and the content of Section \ref{sec:psum}, is to take the sum
over primes inside and apply the prime number theorem in arithmetic
progressions to compute the sum over primes for each fixed $t$. In the process, $L(1, \psi_{t^2-4n})$ gets replaced by $L(1,\widetilde{\psi}_t)$, where $\widetilde{\psi}_t$ is the local average
\[
  \widetilde{\psi}_t(m)
  =
  \frac1{\varphi(m^2)}\sum_{\substack{n\bmod{m^2}\\(n,m)=1}}\psi_{t^2-4n}(m).
\]
This is the content of Lemma \ref{lem:prime_sum}, which, roughly, states that for any nice enough continuous function $\Phi$, we have
\[
    \sum_{\substack{n \in [A,B] \\ n\,\mathrm{prime}}} L(1, \psi_{t^2 - 4n})\Phi(n)\log n
        \approx L(1, \widetilde{\psi}_t) \int_{A}^B \Phi(u)\,du.
\]
Applying this lemma
(after which we will be able to extend the range of summation over $t$ to all
integers) and making a change of variables in the integration
we find for the interval $E = [\alpha_2^{-2}, \alpha_1^{-2}]$ that
\begin{multline*}
\SS = \frac{2h (-1)^{\delta}}{\pi} \sum_{k_0}
    \left(\frac{k_0 - 1}{4\pi}\right)^2
    \int_{\lambda_{k_0} \alpha_1}^{\lambda_{k_0} \alpha_2}
    \sum_{t\in\Z}L(1, \widetilde{\psi}_t)\cos(2\pi\alpha t)\widehat W\!\left(\frac{t}{x_{k_0}(\alpha)}\right)\frac{d\alpha}{\alpha^3} \\
    + O_{E,\varepsilon}\!\left(hK^{2 + \varepsilon} + \frac{HK^{3 + \varepsilon}}{h^{\frac65}}\right),
\end{multline*}
where $\lambda_{k_0} = \frac{k_0 - 1}{4\pi\sqrt{N}}$ and $x_{k_0}(\alpha) = \frac{k_0 - 1}{4\alpha h}$.

The value $L(1, \widetilde{\psi}_t)$ turns out to be a constant times a
multiplicative function (this is part of Lemma \ref{lem:L1_def}), so the sum
over $t$ can be (approximately) computed by studying additive twists of the
Dirichlet series $\sum_{t=1}^\infty L(1, \widetilde{\psi}_t) t^{-s}$,
and we find that the mass of the integrand is concentrated around rational
numbers with squarefree denominators. Specifically, in Proposition
\ref{prop:circle} we find that for $\alpha = a/q + \theta$ (with $\theta$ small),
\[
    \sum_{t\in\Z}L(1, \widetilde{\psi}_t)\cos(2\pi\alpha)\widehat W\!\left(\frac{t}{x}\right)
        \approx \frac{\mu(q)^2}{\varphi(q)^2\sigma(q)} x W(x\theta).
\]

It is now natural to employ the circle method, computing the contribution to
the integral over $[\lambda_{k_0}\alpha_1,\lambda_{k_0}\alpha_2]$ which comes from
major arcs near rational numbers with small denominators, and bounding the
rest. Noticing that $\lambda_{k_0} \approx 1$, we find that the bulk of the
integral comes from rational numbers in the range $[\alpha_1, \alpha_2]$.
The main term will come from a sum of the integrals
\[
    \int x_{k_0}(\alpha)W(x_{k_0}(\alpha)(\alpha - a/q)) \frac{d \alpha}{\alpha^3} \approx \left(\frac{q}{a}\right)^3.
\]
Finally we can see our answer will emerge as
\begin{align*}
\SS &\approx (-1)^\delta \frac{2h}{\pi} \sum_{k_0}
    \left(\frac{k_0-1}{4\pi}\right)^2
    \sum_{\alpha_1 < a/q < \alpha_2} \frac{\mu(q)}{\varphi(q)^2\sigma(q)}\left(\frac{q}{a}\right)^3 \\
    &\approx (-1)^\delta \frac{HK^2}{16\pi^3}
    \sum_{\alpha_1 < a/q < \alpha_2} \frac{\mu(q)}{\varphi(q)^2\sigma(q)}\left(\frac{q}{a}\right)^3,
\end{align*}
though some care must be taken with the endpoints. (See also Remark (\ref{remark:interval}) following Theorem \ref{thm:main}.)
The denominator is easily computed using dimension formulas and the prime number theorem (see \eqref{denominator_total}), giving Theorem \ref{thm:main}.

\section{The sum over weights}\label{sec:ksum}
Let $W:\R\to\R_{\ge0}$ be a smooth, even function supported on $[-1,1]$ and satisfying
\[
W(x)+W(1-x)=1\quad\text{for }x\in[0,1].
\]
 Note that these assumptions imply that $0\le W(x)\le 1$ and $\int_\R W(x)\,dx=1$.
For instance, we may take
\[
W(x)=\begin{cases}
c\int_{-1}^{1-2|x|}e^{-\frac1{1-t^2}}\,dt &\text{ if }|x|<1,\\
0 &\text{ otherwise, }
\end{cases}
\qquad\text{where } c=\left(\int_{-1}^1e^{-\frac1{1-t^2}}\,dt\right)^{-1}=\frac{e^{\frac12}}{K_1(\frac12)-K_0(\frac12)}.
\]
(Here $K_\nu(y)$ is the $K$-Bessel function.)

Define
\[
h=\left\lceil\max(H^{\frac{10}{9}}K^{-\frac19},(HK)^{\frac{5}{11}})\right\rceil
\]
and set
\[
K_0=\min\{k\in\Z:k\ge K-H+4h,\,k\equiv2\delta\bmod4\},
\quad
J=\left\lfloor\frac{K+H-K_0}{4h}\right\rfloor.
\]
We assume $K$ is sufficiently large to ensure that $4h+2\le H\le K-2$, which implies that $K_0\ge4h+2$ and $J>0$.
Write
\begin{equation}\label{e:partition}
\mathbf{1}_{[K-H,K+H]}(k)=
\sum_{j=0}^{J-1}W\!\left(\frac{k-(K_0+4jh)}{4h}\right)
+R(k).
\end{equation}
Then $R$ is supported on $[K-H,K_0]\cup[K_0+4(J-1)h,K+H]$.

Assuming GRH for $L(s, f)$, the sum over primes for a single newform
satisfies the bound
\[
    \sum_{\substack{p\,\mathrm{prime}\\p/N\in E}}\lambda_f(p)\log{p} \ll_{E, \varepsilon} K^{1+\varepsilon}.
\]
Combining this with the estimates $\#H_k(1) \asymp k$ and $R(k) \ll 1$,
we find that
\begin{equation}\label{e:boundary_error}
\sum_{\substack{k\equiv2\delta\bmod4\\|k-K|\le H}} R(k)\sum_{f\in H_k(1)}\sum_{\substack{p\,\mathrm{prime}\\p/N\in E}}\lambda_f(p)\log{p}
\ll_{E,\varepsilon}
hK^{2+\varepsilon},
\end{equation}
as the support of $R(k)$ is of size $O(h)$.

Next we consider the contribution from a typical term of \eqref{e:partition} to \eqref{e:STF_nprime}, with $k_0=K_0+4jh$:
\[
\sum_{k\in k_0+4\Z}\cos((k-1)\phi)
W\!\left(\frac{k-k_0}{4h}\right)
=\Re\sum_{m\in\Z}
W(m/h)e^{i(k_0-1+4m)\phi}.
\]
To that end, we have
\begin{align*}
\int_\R W(x/h)e^{i(k_0-1+4x)\phi}e^{2\pi ixt}\,dx
&= \int_\R W(x/h)
e^{i(k_0-1)\phi+2\pi ix(t+2\phi/\pi)}\,dx\\
&=h\int_\R
W(x)e^{i(k_0-1)\phi+2\pi ix(ht+2h\phi/\pi)}\,dx\\
&=he^{i(k_0-1)\phi}\widehat{W}(ht+2h\phi/\pi),
\end{align*}
and we note that $\widehat{W}$ is real as $W$ is even. By Poisson summation, we obtain
\[
\sum_{k\in k_0+4\Z}\cos((k-1)\phi)
W\!\left(\frac{k-k_0}{4h}\right)
=h\cos((k_0-1)\phi)
\sum_{\ell\in\Z}\widehat{W}(h\ell+2h\phi/\pi).
\]

Applying this to \eqref{e:STF_nprime} with $n$ prime and $k_0\ge4h+2$, we have
\begin{equation}\label{e:Tkn_smooth_k}
\begin{multlined}
\sum_{k\in k_0+4\Z}W\!\left(\frac{k-k_0}{4h}\right)
n^{\frac{1-k}{2}}\Tr T_n|S_k(1)
=-\sum_{k\in k_0+4\Z} 
n^{\frac{1-k}{2}} W\!\left(\frac{k-k_0}{4h}\right)\\
+ \frac{(-1)^{\frac{k_0}{2}}h}{\pi}
\sum_{\ell\in\Z}\sum_{\substack{t\in\Z\\t^2<4n}} 
\widehat{W}\!\left(h\Bigl(\ell+\frac{2\phi_{t,n}}{\pi}\Bigr)\right)
\cos\bigl((k_0-1)\phi_{t,n}\bigr) L(1, \psi_{t^2-4n}).
\end{multlined}
\end{equation}
Since $W$ is smooth, only the $\ell=0$ term contributes significantly to \eqref{e:Tkn_smooth_k}, as quantified by the following lemma.

\begin{lemma}\label{lem:Poisson_error}
We have
\begin{align*}
&\begin{multlined}[t]
\sum_{\substack{n\,\mathrm{prime}\\\frac{n}{N}\in E}}(\log{n})
\Bigg[\frac{(-1)^{\frac{k_0}{2}}h}{\pi}
\sum_{\ell\in\Z\setminus\{0\}}
\sum_{\substack{t\in\Z\\t^2<4n}} 
\cos((k_0-1)\phi_{t,n})L(1, \psi_{t^2-4n})
\widehat{W}\!\left(h\Bigl(\ell+\frac{2\phi_{t,n}}{\pi}\Bigr)\right)\\
-\sum_{k\in k_0+4\Z}n^{\frac{1-k}{2}} W\!\left(\frac{k-k_0}{4h}\right)\Bigg]
\end{multlined}\\
&\hphantom{\sum_{\substack{n\,\mathrm{prime}\\\frac{n}{N}\in E}}}
\ll_E\frac{K^3\log{K}}{h}.
\end{align*}
\end{lemma}

\begin{proof}
First note that
\[
\sum_{\substack{n\,\mathrm{prime}\\\frac{n}{N}\in E}}\log{n}
\sum_{k\in k_0+4\Z}n^{\frac{1-k}{2}} W\!\left(\frac{k-k_0}{4h}\right)
\ll\sum_{n\,\mathrm{prime}}\log{n}
\sum_{k=4}^\infty n^{\frac{1-k}{2}}\ll1.
\]

Next, since $\psi_D$ is periodic modulo $D$ with mean value $0$, partial summation gives the estimate
\[
L(1,\psi_D)=\sum_{m=1}^{|D|}\frac{\psi_D(m)}{m}+\int_{|D|}^\infty x^{-2}\sum_{m\le x}\psi_D(m)\,dx
\ll\sum_{m=1}^{|D|}\frac1{m}+\int_{|D|}^\infty\frac{|D|}{x^2}\,dx 
\ll\log|D|.
\]
Thus,
\[
\sum_{\substack{t\in\Z\\t^2<4n}}\cos((k_0-1)\phi_{t,n})L(1, \psi_{t^2-4n})\ll\sqrt{n}\log{n}.
\]

Since $W$ is smooth, we may apply integration by parts to
$\widehat{W}(x) = \int_\R W(u)e^{-2\pi i xu}\,du$ to obtain
\[
\widehat{W}(x)\ll_A |x|^{-A}
\quad\text{for all }A\in\Z_{\ge0}.
\]
Since $\phi_{t,n}\in(-\frac{\pi}{2},\frac{\pi}{2})$, we have $\big|\ell+\frac{2\phi_{t,n}}{\pi}\big|>|\ell|-1$, so that
\begin{align*}
\sum_{\substack{n\,\mathrm{prime}\\\frac{n}{N}\in E}}
&\frac{h\log{n}}{\pi}
\sum_{\substack{\ell\in\Z\\|\ell|>1}}
\sum_{\substack{t\in\Z\\t^2<4n}} 
\widehat{W}\!\left(h\left(\ell+\frac{2\phi_{t,n}}{\pi}\right)\right) 
\cos((k_0-1)\phi_{t,n})L(1, \psi_{t^2-4n})\\
&\ll
h\sum_{\substack{n\,\mathrm{prime}\\\frac{n}{N}\in E}}
\sqrt{n}\log^2{n}\sum_{\ell=2}^\infty\bigl(h(\ell-1)\bigr)^{-2}
\ll_E\frac{K^3\log{K}}{h}.
\end{align*}

It remains to estimate the contributions from $\ell=\pm1$.
Since $\widehat{W}$ is even, those terms contribute equally, so it suffices to consider $\ell=-1$.
For each fixed $t$, we prove different bounds based on the relative sizes of $t$ and $n$.
When $\ell=-1$ and $t<\sqrt{n}$, $\lvert \ell+2\phi_{t,n}/\pi\rvert\gg1$, and a similar argument as above gives $O_E(\frac{K^3\log{K}}{h})$.

When $\left(\frac{t}{2(1-h^{-2})}\right)^2<n\le t^2$ we use
\[
\frac{\pi}{2}-\phi_{t,n}
= \arcsin\sqrt{1-\frac{t^2}{4n}}
\ge\sqrt{1-\frac{t}{2\sqrt{n}}}
\]
and the estimate $\widehat{W}(x)\ll_A|x|^{-A}$ to get the bound
\begin{align*}
&\ll_{E,A}
h\log^2{K}\sum_{t\asymp_E K}
\int_{(t/2)^2(1-h^{-2})^{-2}}^{t^2}
h^{-A}\left(1-\frac{t}{2\sqrt{n}}\right)^{-A/2}\,d\pi(n)\\
&=h^{1-A}\log^2{K}\sum_{t\asymp_E K}
\int_{L_t}^{t^2}R_{A,t}(n)\,d\pi(n)
\end{align*}
where $L_t=(t/2)^2(1-h^{-2})^{-2}$ and $R_{A,t}(n)=(1-t/(2\sqrt{n}))^{-A/2}$. Applying integration by parts,
\begin{align*}
\int_{L_t}^{t^2} R_{A,t}(n)\,d\pi(n)
&=\bigl(\pi(n)-\pi(t^2/4)\bigr)R_{A,t}(n)\bigr|_{L_t}^{t^2}
-\int_{L_t}^{t^2}\bigl(\pi(n)-\pi(t^2/4)\bigr)R_{A,t}'(n)\,dn\\
&\le2^{A/2}\pi(t^2)-\int_{L_t}^{t^2}\bigl(\pi(n)-\pi(t^2/4)\bigr)R_{A,t}'(n)\,dn.
\end{align*}
Since $L_t-t^2/4\gg_E(K/h)^2$, \cite[\S3.2, Corollary~4]{MV} implies that $\pi(n)-\pi(t^2/4)\ll_E(n-t^2/4)/\log{K}$ for $n\ge L_t$. Hence, for $A\ge3$ we have
\begin{align*}
\int_{L_t}^{t^2} R_{A,t}(n)\,d\pi(n)
&\ll_{E,A}\frac{t^2}{\log{K}}-
\frac1{\log{K}}\int_{L_t}^{t^2}(n-t^2/4)R_{A,t}'(n)\,dn\\
&=\frac{t^2}{\log{K}}-\frac1{\log{K}}
\left((n-t^2/4)R_{A,t}(n)\bigr|_{L_t}^{t^2}
-\int_{L_t}^{t^2}R_{A,t}(n)\,dn\right)\\
&\le\frac{t^2+(L_t-t^2/4)R_{A,t}(L_t)}{\log{K}}
+\frac1{\log{K}}\int_{L_t}^{t^2}R_{A,t}(n)\,dn\\
&\ll_A\frac{t^2h^{A-2}}{\log{K}}.
\end{align*}
Summing over $t\asymp_E K$, this again gives $O_E(\frac{K^3\log{K}}{h})$.

Finally, we estimate the terms with $\left(\frac{t}{2}\right)^2<n\le\left(\frac{t}{2(1-h^{-2})}\right)^2$ using the trivial bound $\widehat{W}(x)\ll1$. Summing over $t\asymp_E K$, these make a total contribution
\[
\ll_E h\log^2{K}\sum_{t\asymp_E K}\left[\pi\!\left(\left(\frac{t}{2(1-h^{-2})}\right)^2\right)-\pi\!\left(\left(\frac{t}{2}\right)^2\right)\right].
\]
Note that $\left(\frac{t}{2(1-h^{-2})}\right)^2=\left(\frac{t}{2}\right)^2+O_E(K^2/h^2)$.
Applying~\cite[\S3.2, Corollary~4]{MV} again, it follows that
\[
\pi\!\left(\left(\frac{t}{2(1-h^{-2})}\right)^2\right)-\pi\!\left(\left(\frac{t}{2}\right)^2\right)
\ll_E\frac{K^2}{h^2\log{K}},
\]
so in total we get $\ll_E\frac{K^3\log{K}}{h}$.
\end{proof}

Summing the bound from Lemma~\ref{lem:Poisson_error} over the $J\ll H/h$ choices of $k_0$, and combining with the boundary error~\eqref{e:boundary_error} from the partition of unity, we get that the total contribution from everything except the $\ell = 0$ term is bounded by
\[
\ll_{E,\varepsilon} hK^{2+\varepsilon}+\frac{HK^3\log{K}}{h^2}.
\]

Next we fix a small $\varepsilon_0>0$ and
restrict the sum over $t$ to $|t|\le T=K^{1+\varepsilon_0}/h$. The error in so doing is $O_{E,\varepsilon_0}(1)$, thanks to the rapid decay of $\widehat{W}$. Hence, we have
\begin{multline*}
\Sigma = 
O_{E,\varepsilon_0,\varepsilon}
\!\left(hK^{2+\varepsilon}+\frac{HK^3\log{K}}{h^2}\right)\\
+\frac{(-1)^{\delta}h}{\pi}
\sum_{\substack{k_0\equiv K_0\bmod{4h}\\0\le\frac{k_0-K_0}{4h}<J}}
\sum_{\substack{n\,\mathrm{prime}\\\frac{n}{N}\in E}}\log{n}\sum_{\substack{t\in\Z\\|t|\le T}}
L(1,\psi_{t^2-4n})\cos((k_0-1)\phi_{t,n})
\widehat{W}\!\left(h\frac{2\phi_{t,n}}{\pi}\right).
\end{multline*}
Writing $A_{k_0}(\phi)=\cos((k_0-1)\phi)\widehat{W}(2h\phi/\pi)$, we have $A_{k_0}'(\phi)\ll K$ uniformly, and thus
\[
A_{k_0}(\phi_{t,n})=A_{k_0}\!\left(\frac{t}{2\sqrt{n}}\right)+O_E\!\left(\frac{T^3}{K^2}\right),
\]
since
\[
\phi_{t,n}=\frac{t}{2\sqrt{n}}+O\!\left(\frac{|t|^3}{n^{\frac32}}\right).
\]
Summing this error over $|t|\le T$, $n\asymp_E K^2$, and $k_0$, we get
\[
\ll_E HT^4\log{K}
=\frac{HK^{4+4\varepsilon_0}\log{K}}{h^4},
\]
which is dominated by the above error term for $\varepsilon_0$ sufficiently small. Thus, we have
\begin{equation}\label{e:ksum_result}
\begin{multlined}
\SS=O_{E,\varepsilon_0,\varepsilon}
\!\left(hK^{2+\varepsilon}+\frac{HK^3\log{K}}{h^2}\right)\\
+\frac{(-1)^{\delta}h}{\pi}
\sum_{\substack{k_0\equiv K_0\bmod{4h}\\0\le\frac{k_0-K_0}{4h}<J}}
\sum_{\substack{t\in\Z\\|t|\le T}}
\sum_{\substack{n\,\mathrm{prime}\\\frac{n}{N}\in E}}
(\log{n})L(1,\psi_{t^2-4n})
\cos\!\left(\frac{(k_0-1)t}{2\sqrt{n}}\right)
\widehat{W}\!\left(\frac{ht}{\pi\sqrt{n}}\right).
\end{multlined}
\end{equation}

\section{The sum over primes}\label{sec:psum}
Next we turn our attention to the sum over primes $n$.
We begin with some lemmas.
\begin{lemma}\label{lem:psiD_modm^2}
Let $m$ be a positive integer, and let $D_1,D_2$ be discriminants with $D_1\equiv D_2\bmod{(2m)^2}$. Then $\psi_{D_1}(m)=\psi_{D_2}(m)$. 
\end{lemma}
\begin{proof}
Write $D_i=d_i\ell_i^2$ where $d_i$ is a fundamental discriminant and $\ell_i\in\Z_{\ge0}$. Without loss of generality, we assume that $\ord_2(\ell_1)\leq \ord_2(\ell_2)$. Let $r = \gcd(\ell_1, m)$.
Since $r^2$ divides both $D_1$ and $m^2$, it also divides $D_2$.
Moreover, since $d_2/\gcd(d_2,4)$ is squarefree, we have $r\mid2\ell_2$. 
Since $\ord_2(\ell_1)\leq\ord_{2}(\ell_2)$, it follows that $r\mid\ell_2$.

Suppose $p$ is a prime dividing $\gcd(\frac{\ell_2}{r},\frac{m}{r})$. Then $p\nmid\frac{\ell_1}{r}$ since $\gcd(\frac{\ell_1}{r},\frac{m}{r})=1$. From $d_1(\ell_1/r)^2\equiv d_2(\ell_2/r)^2\bmod{(2m/r)^2}$ it follows that $p^2\mid d_1$, whence $p=2$. Dividing by $4$, we find that $(d_1/4)(\ell_1/r)^2\equiv d_2(\ell_2/(2r))^2\bmod4$, which is impossible since $(d_1/4)(\ell_1/r)^2\equiv2\text{ or }3\bmod4$ and $d_2(\ell_2/(2r))^2\equiv0\text{ or }1\bmod4$.

Hence $r=\gcd(\ell_2,m)$, and
\[
\psi_{D_1}(m)
=\left(\frac{d_1}{m/r}\right)
=\left(\frac{d_1(\ell_1/r)^2}{m/r}\right)
=\left(\frac{d_2(\ell_2/r)^2}{m/r}\right)
=\left(\frac{d_2}{m/r}\right)
=\psi_{D_2}(m).
\]
\end{proof}

\begin{lemma}\label{lem:multiplicative}
For $t,m\in\Z$ with $m>0$, define
\[
\widetilde{\psi}_t(m)=\frac1{\varphi(m^2)}\sum_{\substack{n\bmod{m^2}\\(n,m)=1}}\psi_{t^2-4n}(m).
\]
Then for fixed $t$, $m\mapsto\widetilde{\psi}_t(m)$ is a multiplicative function of $m$.
\end{lemma}
\begin{proof}
Fix $t\in \Z$, and consider $m_1, m_2\in \Z_{>0}$ with $\gcd(m_1, m_2)=1$. 
For $n\in \Z$, we write $t^2-4n=d\ell^2$ where $d$ is a fundamental discriminant and $\ell\in \Z_{\geq 0}$.
Let $r_1=\gcd(m_1, \ell)$ and $r_2=\gcd(m_2, \ell)$.
Since $\gcd(m_1, m_2)=1$, we have $\gcd(r_1, r_2)=1$ and $\gcd(m_1m_2, \ell)=r_1r_2$.
Therefore,
\[
\psi_{t^2-4n}(m_1m_2) = \left(\frac{d}{\frac{m_1}{r_1} \frac{m_2}{r_2}}\right)
= \left(\frac{d}{m_1/r_1}\right)\left(\frac{d}{m_2/r_2}\right)
= \psi_{t^2-4n}(m_1) \psi_{t^2-4n}(m_2),
\]
so that
\begin{multline*}
\widetilde{\psi}_t(m_1m_2)
= \frac1{\varphi(m_1m_2)}\sum_{\substack{n\bmod{(m_1m_2)^2}\\(n, m_1m_2)=1}} \psi_{t^2-4n}(m_1)\psi_{t^2-4n}(m_2)
\\ = \frac1{\varphi(m_1m_2)}\sum_{\substack{n_1\bmod{m_1^2}\\(n_1, m_1)=1}} \sum_{\substack{n_2\bmod{m_2^2}\\(n_2, m_2)=1}} 
\psi_{t^2-4(n_1(m_2\overline{m}_2)^2+n_2(m_1\overline{m}_1)^2)}(m_1)
\psi_{t^2-4(n_1(m_2\overline{m}_2)^2+n_2(m_1\overline{m}_1)^2)}(m_2),
\end{multline*}
where $m_1\overline{m}_1\equiv1\bmod{m_2^2}$ and $m_2\overline{m}_2\equiv1\bmod{m_1^2}$.
By Lemma~\ref{lem:psiD_modm^2} we have
\[
\psi_{t^2-4(n_1(m_2\overline{m}_2)^2+n_2(m_1\overline{m}_1)^2)}(m_i)
=\psi_{t^2-4n_i}(m_i),
\]
so we get
\[
\widetilde{\psi}_t(m_1m_2)= \widetilde{\psi}_t(m_1)\widetilde{\psi}_t(m_2). 
\]
\end{proof}

\begin{lemma}\label{lem:L1_def}
For $t\in\Z$, define
\[
L(s,\widetilde{\psi}_t)=\sum_{m=1}^\infty\frac{\widetilde{\psi}_t(m)}{m^s}\quad\text{for }\Re(s)>1.
\]
Then 
\[
L(s, \widetilde{\psi}_t)
= L(s,\widetilde{\psi}_1)P(s,t),
\]
where
\begin{equation}\label{e:Lspsi1_def}
L(s,\widetilde{\psi}_1)=\zeta(2s)\zeta(s+2)
\prod_p
\begin{cases}
1-(1+p^{-1})(1+p^{-s})p^{-s-1}&p>2,\\
(1-2^{-s})(1-2^{-s-2})&p=2,
\end{cases}
\end{equation}
and
\begin{equation}\label{e:Pst_def}
P(s,t)=
\prod_{p\mid t}
\begin{cases}
\frac{1-p^{-s-2}}{1-(1+p^{-1})(1+p^{-s})p^{-s-1}}&p>2,\\
\frac{1+2^{-s-1}-2^{-2s}}{1-2^{-s}}&p=2,\;4\mid t,\\
\frac{1+2^{-s-2}-7\cdot2^{-2s-3}-2^{-3s-2}}{(1-2^{-s})(1-2^{-s-2})}&p=2,\;4\nmid t.
\end{cases}
\end{equation}
Thus $L(s,\widetilde{\psi}_t)$ continues analytically to $\Re(s)>\frac12$ and satisfies
\[
L(1,\widetilde{\psi}_t)=Cf(t),
\]
where
\[
C=L(1,\widetilde{\psi}_1)=\prod_p\left(1-\frac1{(p-1)^2(p+1)}\right)
=0.6151326573181718\ldots
\]
and
\begin{equation}\label{e:f_def}
f(t)=P(1,t)=\prod_{p\mid t}\left(1+\frac1{p^2-p-1}\right).
\end{equation}
In particular, 
\[
    L(1, \widetilde{\psi}_0) = Cf(0) = \prod_p\left(1 - \frac{1}{(p-1)^2(p+1)}\right)\left(1 + \frac{1}{p^2 - p - 1}\right) = \prod_p \left(1 - \frac{1}{p^2}\right)^{-1} = \zeta(2).
\]
\end{lemma}

\begin{proof}
Fix $t\in\Z$. By Lemma~\ref{lem:multiplicative}, $\widetilde{\psi}_t$ is multiplicative, so we have $L(s,\widetilde{\psi}_t)=\prod_p\sum_{e=0}^\infty \widetilde{\psi}_t(p^e)p^{-es}$.

Consider $m=p^e$ for some $e>0$, and put
\[
S=\sum_{\substack{n\bmod{p^{2e}}\\(n,p)=1}}
\psi_{t^2-4n}(p^e).
\]
Suppose first that $p\nmid 2t$. Given a coprime residue $n\bmod{p^{2e}}$, we put $j=\ord_p(t^2-4n)$.
The term $n\equiv(t/2)^2\bmod{p^{2e}}$ corresponds to $j\ge2e$ and contributes $\psi_{t^2-4n}(m)=1$ to $S$.

Consider $j\in(0,2e)$. Then $n$ is of the form
$(t/2)^2-(a+bp)p^j$, where $a\in\{1,\ldots,p-1\}$ and $b\in\{0,\ldots,p^{2e-j-1}-1\}$.
We have
\[
\psi_{t^2-4n}(m) =
\begin{cases}
\left(\frac{a}{p^{e-j/2}}\right)
&j\text{ even},\\
0&j\text{ odd}.
\end{cases}
\]
From even $j=2e-2r$ we get a contribution of
\[
\sum_{a=1}^{p-1}\sum_{b=0}^{p^{2r-1}-1}
\left(\frac{a}{p^r}\right)
=\begin{cases}
\varphi(p^{2r})&r\text{ even},\\
0&r\text{ odd}.
\end{cases}
\]

Now consider $j=0$. As before we have $n=(t/2)^2-(a+bp)$ with $a\in\{1,\ldots,p-1\}$ and $b\in\{0,\ldots,p^{2e-1}-1\}$, but now we have the additional constraint $a\not\equiv(t/2)^2\bmod p$, since $n$ has to be coprime to $p$. We have $\psi_{t^2-4n}(m)=\left(\frac{a}{p^e}\right)$, so the contribution from $j=0$ is
\[
\sum_{\substack{1\le a\le p-1\\a\not\equiv(t/2)^2\bmod{p}}}
\sum_{b=0}^{p^{2e-1}-1}
\left(\frac{a}{p^e}\right)
=\begin{cases}
(p-2)p^{2e-1}&e\text{ even},\\
-p^{2e-1}&e\text{ odd}.
\end{cases}
\]

Altogether, for $p\nmid2t$ and $e>0$ we have
\[
S=\sum_{\substack{n\bmod{p^{2e}}\\(n,p)=1}}
\psi_{t^2-4n}(p^e)
=-p^{2e-1}+\sum_{k=0}^{\lfloor{e/2}\rfloor}\varphi(p^{4k}),
\]
so that
\begin{align*}
\sum_{e=0}^\infty\frac{\widetilde{\psi}_t(p^e)}{p^{es}}
&=-\sum_{e=1}^\infty\frac{p^{2e-1}}{p^{es}\varphi(p^{2e})} + \sum_{k=0}^\infty\sum_{e=2k}^\infty\frac{\varphi(p^{4k})}{p^{es}\varphi(p^{2e})}\\
&=\frac{1-(1+p^{-1})(1+p^{-s})p^{-s-1}}{(1-p^{-2s})(1-p^{-s-2})}.
\end{align*}

Next suppose that $p=2$ and $t$ is odd. Then $j=0$, so we get
\[
S=\sum_{\substack{n\bmod{2^{2e}}\\(n,2)=1}}
\left(\frac{t^2-4n}{2^e}\right)=
(-1)^e2^{2e-1}.
\]
Thus,
\[
\sum_{e=0}^\infty\frac{\widetilde{\psi}_t(2^e)}{2^{es}}=
1+\sum_{e=1}^\infty\frac{(-1)^e2^{2e-1}}{2^{es}\varphi(2^{2e})}
=\frac1{1+2^{-s}}.
\]
Hence,
\[
L(s,\widetilde{\psi}_1)=\frac1{1+2^{-s}}
\prod_{p>2}\frac{1-(1+p^{-1})(1+p^{-s})p^{-s-1}}{(1-p^{-2s})(1-p^{-s-2})},
\]
which yields \eqref{e:Lspsi1_def}.

Now suppose that $p\mid t$. When $p>2$ we have $j=0$, so that
$\widetilde{\psi}_t(p^e)=(1+(-1)^e)/2$,
and the local factor at $p$ becomes $1/(1-p^{-2s})$.
When $p=2$ and $4\mid t$, we have $\widetilde{\psi}_t(2^e)=(1-(-1)^e)/4$ for $e>0$, so that
\begin{align*}
\sum_{e=0}^\infty \widetilde{\psi}_t(2^{e})2^{-es}
= 1+\frac{1}{4}\sum_{e=1}^\infty (1-(-1)^e) 2^{-es}
= \frac{1+2^{-1-s}-2^{-2s}}{1-2^{-2s}}.
\end{align*}
When $2\mid t$ but $4\nmid t$, we have 
$\widetilde{\psi}_t(2^e)=2^{1-2e}\bigl(1+\frac{16^{\lfloor{e/2}\rfloor}-1}{15}\bigr)$ for $e>0$, so that
\begin{align*}
\sum_{e=0}^\infty \widetilde{\psi}_t(2^{e})2^{-es}
&= 1+ \frac{14\cdot 2}{15}\sum_{e=1}^\infty 2^{-2e}2^{-es} + \frac{2}{15} \sum_{e=1}^\infty 2^{-2e}2^{4\lfloor\frac{e}{2}\rfloor} 2^{-es}\\
&=\frac{1+2^{-s-2}- 7\cdot 2^{-2s-3} - 2^{-3s-2}}{(1-2^{-s-2})(1-2^{-2s})}.
\end{align*}
Collecting these cases, we arrive at \eqref{e:Pst_def}.
\end{proof}

\begin{lemma}\label{lem:L1D}
Assume the Generalized Lindel\"of Hypothesis (GLH) for quadratic Dirichlet $L$-functions. Then for any non-square discriminant $D$ and any $x\ge1$,
\[
L(1,\psi_D)=\sum_{m\le x}\frac{\psi_D(m)}{m}
+O_\varepsilon\!\left(\frac{|Dx|^\varepsilon}{\sqrt{x}}\right).
\]
\end{lemma}
\begin{proof}
Let $D=d\ell^2$ with $d$ a fundamental discriminant and $\ell\in\Z_{>0}$. Then we have
\[
L(s,\psi_D):=\sum_{m=1}^\infty\frac{\psi_D(m)}{m^s}=L(s,\psi_d)\prod_{p\mid\ell}
\bigg(1+(1-\psi_d(p))\sum_{j=1}^{\ord_p\ell}p^{-js}\bigg).
\]
For $\Re(s)\ge\frac12$ we have
\[
\left|1+(1-\psi_d(p))\sum_{j=1}^{\ord_p\ell}p^{-js}\right|
\le1+2\sum_{j=1}^\infty p^{-j/2}
=\frac{\sqrt{p}+1}{\sqrt{p}-1},
\]
and hence
\[
L(s,\psi_D)\ll_\varepsilon|L(s,\psi_d)|\ell^\varepsilon\ll_\varepsilon|Ds|^\varepsilon,
\]
under GLH. By Perron's formula, for $x\ge2$,
\[
\sum_{n\le x}\frac{\psi_D(m)}{m}
=\frac1{2\pi i}\int_{1+\frac1{\log{x}}-ix}^{1+\frac1{\log{x}}+ix}
L(s,\psi_D)x^{s-1}\frac{ds}{s-1}+O\!\left(\frac{\log{x}}{x}\right).
\]
The lemma follows on shifting the contour to $\Re(s)=\frac12$.
\end{proof}

\begin{lemma}\label{lem:prime_sum}
Assume GRH for Dirichlet $L$-functions.
Let $t\in\Z$ and $A,B\in\R$ with $\frac{t^2}{4}<A<B$, and let $\Phi\in C^1([A,B])$. Set $M=\max_{u\in[A,B]}\lvert \Phi(u) \rvert$ and $V=\int_A^B\lvert \Phi'(u) \rvert\,du$. Then
\[
\sum_{\substack{n\in[A,B]\\n\,\mathrm{prime}}}L(1,\psi_{t^2-4n})\Phi(n)\log{n}
=L(1,\widetilde{\psi}_t)\int_A^B\Phi(u)\,du+O_\varepsilon\bigl(M^{\frac45}(M+V)^{\frac15}B^{\frac{9}{10}+\varepsilon}\bigr)
\quad\forall\varepsilon\in(0,\tfrac1{10}].
\]
\end{lemma}
\begin{proof}
The result is trivially true if $B<2$, so assume otherwise.
Put $I=\{n\in(A,B]:n\,\mathrm{prime}\}$. (Note that $I$ omits the left endpoint.)
For each positive integer $m$, we have
\begin{align*}
\sum_{n\in I}\psi_{t^2-4n}(m)\Phi(n)\log{n}
& = \sum_{\substack{a\bmod{m^2}\\ \gcd(a, m)=1}} \sum_{\substack{n\in I\\ n\equiv a\bmod{m^2}}}\psi_{t^2-4n}(m)\Phi(n)\log{n} 
+ \sum_{\substack{n\in I\\n\mid m}} \psi_{t^2-4n}(m)\Phi(n)\log{n}
\\ & = \sum_{\substack{a\bmod{m^2}\\ \gcd(a,m)=1}} \psi_{t^2-4a}(m) \sum_{\substack{n\in I\\n\equiv a\bmod{m^2}}}\Phi(n)\log{n}
+ \sum_{\substack{n\in I\\n\mid m}} \psi_{t^2-4n}(m)\Phi(n)\log{n}.
\end{align*}
Here for the first piece we use Lemma \ref{lem:psiD_modm^2}, which states that $\psi_{t^2-4n}(m)=\psi_{t^2-4a}(m)$ if $a\equiv n\bmod{m^2}$. 
For the second piece, we observe
\[
\biggl\lvert\sum_{\substack{n\in I\\n\mid m}} \psi_{t^2-4n}(m)\Phi(n)\log{n}\biggr\rvert
\leq M\log{m} \leq M\varphi(m^2).
\]
By~\cite[\S13.1, Theorem~8]{MV}, GRH implies that for $m\geq 1$, $(a, m)=1$ and $x\geq 2$, 
\[
\theta(x; m^2, a) :=
\sum_{\substack{p\le x\\p\equiv a\bmod{m^2}}}\log{n}
= \frac{x}{\varphi(m^2)} + O\bigl(x^{\frac{1}{2}}\log^2x\bigr).
\]
Hence, writing $E(x;m^2,a)=\theta(x;m^2,a)-\frac{x}{\varphi(m^2)}$, we have 
\[
\sum_{\substack{n\in I\\n\equiv a\bmod{m^2}}}\Phi(n)
=\int_A^B\Phi(u)\,d\theta(u;m^2,a)
=\int_A^B\frac{\Phi(u)}{\varphi(m^2)}\,du
+\int_A^B\Phi(u)\,dE(u;m^2,a).
\]
Applying integration by parts, the error term is
\[
\Phi(u)E(u;m^2,a)\Bigr|_A^B
-\int_A^BE(u;m^2,a)\Phi'(u)\,du
\ll(M+V)B^{\frac12}\log^2{B},
\]
so that
\[
\sum_{n\in I}\psi_{t^2-4n}(m)\Phi(n)\log{n}
= \widetilde{\psi}_t(m)\int_A^B\Phi(u)\,du + O\bigl(\varphi(m^2)(M+V)B^{\frac{1}{2}}\log^2{B}\bigr).
\]
By Lemma \ref{lem:L1D}, 
\begin{align*}
\sum_{n\in I}{}&L(1,\psi_{t^2-4n})\Phi(n)\log{n}
= \sum_{m\leq x}\frac{1}{m}\sum_{n\in I}\psi_{t^2-4n}(m)\Phi(n)\log{n}
+ \sum_{n\in I}|\Phi(n)\log{n}|O_\varepsilon\!\left(\frac{|(t^2-4n)x|^\varepsilon}{\sqrt{x}}\right)\\
& = \sum_{m\leq x}\frac{\widetilde{\psi}_t(m)}{m}
\int_A^B\Phi(u)\,du
+ O\bigg((M+V)B^{\frac12}\log^2{B}\sum_{m\le x}\frac{\varphi(m^2)}{m}\bigg)
+ O_\varepsilon\bigl(MB^{1+\varepsilon}x^{-\frac12+\varepsilon}\bigr)\\
&=\sum_{m\leq x}\frac{\widetilde{\psi}_t(m)}{m}
\int_A^B\Phi(u)\,du
+ O\bigl(x^2(M+V)B^{\frac12}\log^2{B}\bigr)
+ O_\varepsilon\bigl(MB^{1+\varepsilon}x^{-\frac12+\varepsilon}\bigr).
\end{align*}

From the proof of Lemma~\ref{lem:L1_def} we may observe that $|\widetilde{\psi}_t(p)|\le\frac1p$ for primes $p>2$, which implies the estimate
\[
\bigl|\widetilde{\psi}_t(m)\bigr|
\le2\prod_{\substack{p\mid m\\p^2\nmid m}}\frac1p.
\]
Given $m>x$, we may write $m=dr$ where $d$ is squarefree, $r$ is squarefull, and $(d,r)=1$. Thus,
\[
\sum_{m>x}\frac{|\widetilde{\psi}_t(m)|}{m}
\le2\sum_{d=1}^\infty\frac{\mu^2(d)}{d^2}
\sum_{\substack{r\text{ squarefull}\\r>x/d}}\frac1r
\ll\sum_{d=1}^\infty\frac{\min(1,\sqrt{d/x})}{d^2}\ll\frac1{\sqrt{x}}.
\]
Thus, we have
\[
\sum_{n\in I}L(1,\psi_{t^2-4n})\Phi(n)\log{n}
=L(1,\widetilde{\psi}_t)\int_A^B\Phi(u)\,du
+ O\bigl(x^2(M+V)B^{\frac12}\log^2{B}\bigr)
+ O_\varepsilon\bigl(MB^{1+\varepsilon}x^{-\frac12+\varepsilon}\bigr).
\]
Finally, if $A$ is a prime number, we absorb the remaining term $L(1,\psi_{t^2-4A})\Phi(A)\log{A}$ into the error term.
The lemma follows on choosing $x=\bigl(B^{\frac12}M/(M+V)\bigr)^{\frac15}$.
\end{proof}

With these results in place, we can now compute the sum over $n$ in \eqref{e:ksum_result}.
For fixed $k_0$ and $t$, let
\[
\Phi_{k_0,t}(u)=\cos\!\left(\frac{(k_0-1)t}{2\sqrt{u}}\right)
\widehat{W}\!\left(\frac{ht}{\pi\sqrt{u}}\right).
\]
Then
\[
\Phi_{k_0,t}(u)\ll1
\quad\text{and}\quad
\int_{NE}\bigl|\Phi_{k_0,t}'(u)\bigr|\,du\ll_E|t|.
\]
Summing the error term from Lemma~\ref{lem:prime_sum} over $|t|\le T$ and $k_0$, we get
\begin{equation}\label{e:sum_t<T}
\ll_{E,\varepsilon_0,\varepsilon}T^{\frac65}K^{\frac95+\varepsilon}H
=\frac{K^{3+\frac65\varepsilon_0+\varepsilon}H}{h^{\frac65}}.
\end{equation}
The sum of the error terms from \eqref{e:ksum_result} and \eqref{e:sum_t<T} is thus
\[
\ll_{E,\varepsilon_0,\varepsilon}
hK^{2+\varepsilon}
+\frac{HK^3\log{K}}{h^2}
+\frac{HK^{3+\frac65\varepsilon_0+\varepsilon}}{h^{\frac65}}
\ll hK^{2+\varepsilon}
+\frac{HK^{3+\frac65\varepsilon_0+\varepsilon}}{h^{\frac65}},
\]
so we obtain
\[
\SS=O_{E,\varepsilon_0,\varepsilon}\!\left(hK^{2+\varepsilon}
+\frac{HK^{3+\frac65\varepsilon_0+\varepsilon}}{h^{\frac65}}\right)
+\frac{(-1)^{\delta}h}{\pi}
\sum_{k_0}\int_{NE}\sum_{\substack{t\in\Z\\|t|\le T}} 
L(1,\widetilde{\psi}_t)
\cos\!\left(\frac{(k_0-1)t}{2\sqrt{u}}\right)
\widehat{W}\!\left(\frac{ht}{\pi\sqrt{u}}\right)
du.
\]

Having replaced $L(1,\psi_{t^2-4n})$ by its average, we can now extend the sum over $t$ out to $\pm\infty$, with an error of $O_{E,\varepsilon_0}(1)$. Taking $\varepsilon_0$ arbitrarily small, writing $E=[\alpha_2^{-2},\alpha_1^{-2}]$ and making the substitution $u=\bigl(\frac{k_0-1}{4\pi\alpha}\bigr)^2$, we thus get
\begin{equation}\label{e:psum_result}
\begin{multlined}
\SS=O_{E,\varepsilon}\!\left(hK^{2+\varepsilon}
+\frac{HK^{3+\varepsilon}}{h^{\frac65}}\right)\\
+\frac{(-1)^{\delta}2h}{\pi}\sum_{k_0}
\left(\frac{k_0-1}{4\pi}\right)^2
\int_{\lambda_{k_0}\alpha_1}^{\lambda_{k_0}\alpha_2}\sum_{t\in\Z} 
L(1,\widetilde{\psi}_t)
\cos(2\pi\alpha t)
\widehat{W}\!\left(\frac{t}{x_{k_0}(\alpha)}\right)
\frac{d\alpha}{\alpha^3},
\end{multlined}
\end{equation}
where $\lambda_{k_0}=\frac{k_0-1}{4\pi\sqrt{N}}$ and $x_{k_0}(\alpha)=\frac{k_0-1}{4\alpha h}$.

\section{The exponential sum}\label{sec:tsum}
In this section we evaluate the sum over $t$ in \eqref{e:psum_result}, proving the following:
\begin{proposition}\label{prop:circle}
Assume GRH for Dirichlet $L$-functions.
Let $\alpha,\theta,x\in\R$ and $a,q\in\Z$ with $x,q\ge1$, $\gcd(a,q)=1$, $\alpha=\frac{a}{q}+\theta$, and $|\theta|\le\frac1{q^2}$.
Then,
\begin{align*}
\sum_{t\in\Z}L(1,\widetilde{\psi}_t)\cos(2\pi\alpha t) \widehat{W}\!\left(\frac{t}{x}\right)
&=\frac{\mu(q)^2}{\varphi(q)^2\sigma(q)}xW(x\theta)\\
&+O\bigl(qx^{-1}\max(1,x|\theta|)\bigr)+O_\varepsilon\bigl(q^3x^{-\frac74+\varepsilon}\max(1,x|\theta|)^{\frac72}\bigr).
\end{align*}
\end{proposition}

Recall from Lemma~\ref{lem:L1_def} that $L(1, \widetilde{\psi}_t) = C f(t)$, where $C = L(1, \widetilde{\psi}_1)$ and
\begin{equation*}
    f(t) = \prod_{p \mid t} \biggl( 1 - \frac{1}{(p-1)^2(p+1)} \biggr)
\end{equation*}
is multiplicative.
The next lemma studies the generating function of $f$ and its additive twists.
\begin{lemma}\label{lem:F}
Let $q\in\Z_{>0}$ and $a\in\Z$ with $\gcd(a, q)=1$.
For $\Re(s)>1$, define
\[
F(s; a/q) = \sum_{n=1}^\infty \frac{f(n)e(an/q)}{n^s}
\quad\text{and}\quad
F^\pm(s;a/q)=\frac{F(s;a/q)\pm F(s;-a/q)}{2}.
\]
Then 
\begin{equation}\label{e:Fsa/q_sumofLchi}
F^\pm(s; a/q) = \sum_{d\mid q}\frac{f(d)}{d^{s}}
\sum_{\substack{\chi\bmod{\frac{q}{d}}\\\chi(-1)=\pm1}} \frac{\chi(a)\tau(\overline\chi)}{\varphi(q/d)}
\frac{L(s, \chi)L(s+2,\chi\chi_0)}{L(2(s+2),\chi\chi_0)}
\prod_{p\nmid q}\left(1 + \frac{\chi(p)p^{-s-2}\frac{p+1}{p^2-p-1}}{1 + \chi(p)p^{-s-2}}\right),
\end{equation}
where $\chi_0$ denotes the trivial character modulo $q$.

It follows that $F^\pm(s;a/q)$ has meromorphic continuation to $\Re(s)>-2$ with the following additional properties:
\begin{itemize}
\item $F^-(s;a/q)$ and $(s^2-1)F^+(s;a/q)$ are analytic for $\Re(s)\ge-\frac32$ unconditionally, and for $\Re(s)>-\frac74$ under GRH;
\item $\Res_{s=1} F^+(s; a/q)= \frac{1}{C} \frac{\mu(q)^2}{\varphi(q)^2\sigma(q)}$;
\item $F^+(0;a/q) = -\tfrac12f(0)$;
\item $\Res_{s=-1}F^+(s;a/q)=-\frac{q}{12C}\prod_{p\parallel q}\left(1-\frac1{p^3-p}\right)$; 
\item $F^-(-1;a/q)=0$.
\end{itemize}
\end{lemma}
\begin{proof}
By orthogonality of Dirichlet characters,
\[
e\!\left(\frac{an}{q}\right)
= \sum_{d\mid\gcd(n,q)} \frac{1}{\varphi(q/d)} \sum_{\chi\bmod{\frac{q}{d}}} \chi(a)\tau(\overline\chi)\chi\!\left(\frac{n}{d}\right).
\]
For $\Re(s)>1$, we have 
\begin{align*}
F(s; a/q)
& = \sum_{n=1}^\infty \frac{f(n)e(an/q)}{n^s}
= \sum_{n=1}^\infty \sum_{d\mid\gcd(n, q)} 
\frac{1}{\varphi(q/d)} \sum_{\chi\bmod{\frac{q}{d}}} \chi(a) \tau(\overline\chi)
\frac{f(n)\chi(n/d)}{n^s}
\\ & = \sum_{d\mid q}\frac{1}{\varphi(q/d)} 
\sum_{\chi\bmod{\frac{q}{d}}}
\chi(a)\tau(\overline\chi) 
\sum_{n=1}^\infty \frac{f(nd)\chi(n)}{(nd)^{s}}. 
\end{align*}
Recalling \eqref{e:f_def}, we observe that $f(n)$ is multiplicative; further, for any $j\geq 1$, 
\[
f(p^j)= f(p) = 1+\frac1{p^2-p-1}. 
\]
So for each $d\mid q$ and $\chi$ modulo $\frac{q}{d}$, we get
\begin{align*}
\sum_{n=1}^\infty \frac{f(nd)\chi(n)}{(nd)^{s}}
& = d^{-s} \prod_{p} \sum_{j=0}^\infty \frac{f(p^{{\rm ord}_p(d)+j}) \chi(p^j)}{p^{js}}
\\ & = d^{-s} 
\prod_{p\mid d} \bigg(f(p) \sum_{j=0}^\infty \frac{\chi(p^j)}{p^{js}}\bigg)
\prod_{p\nmid d} \bigg(1+f(p)\sum_{j=1}^\infty \frac{\chi(p^j)}{p^{js}}\bigg)
\\ & = d^{-s} f(d) L(s, \chi)
\prod_{p\nmid q}\bigl(1+(f(p)-1)\chi(p)p^{-s}\bigr).
\end{align*}
Thus we have
\begin{equation}\label{e:Fpm_first_expression}
F^\pm(s; a/q)
= \sum_{d\mid q}\frac{f(d)}{d^{s}}
\frac{1}{\varphi(q/d)}\sum_{\substack{\chi\bmod{\frac{q}{d}}\\\chi(-1)=\pm1}} \chi(a)\tau(\overline\chi) 
L(s, \chi)
\prod_{p\nmid q}\bigl(1+(f(p)-1)\chi(p)p^{-s}\bigr).
\end{equation}
Further, we can write
\[
\prod_{p\nmid q}(1 + (f(p)-1)\chi(p)p^{-s}) = \frac{L(s+2,\chi\chi_0)}{L(2(s+2),\chi\chi_0)}
\prod_{p\nmid q}\left(1 + \frac{\chi(p)p^{-s-2}(p+1)/(p^2-p-1)}{1 + \chi(p)p^{-s-2}}\right),
\]
where $\chi_0$ denotes the trivial character mod $q$, which in turn implies \eqref{e:Fsa/q_sumofLchi}. It follows that $F^\pm(s;a/q)$ has meromorphic continuation to $\Re(s)>-2$, with poles possible at the poles of $\frac{L(s,\chi)L(s+2,\chi\chi_0)}{L(2(s+2),\chi\chi_0)}$. For $\Re(s)\ge-\frac32$ unconditionally, and $\Re(s)>-\frac74$ under GRH, poles can only occur at $s=\pm1$ when $\chi$ is the trivial character modulo $\frac{q}{d}$ for some $d\mid q$. Since the trivial character is always even, these poles can only occur in $F^+$.

Let $\chi_{q/d}$ denote the trivial character modulo $\frac{q}{d}$. Then
\[
\tau(\overline{\chi}_{q/d}) = \sum_{\substack{u\bmod{\frac{q}{d}}\\ \gcd(u,\frac{q}{d})=1}} e^{2\pi i \frac{u}{q/d}}
= \mu\!\left(\frac{q}{d}\right),
\]
so that
\begin{align*}
\Res_{s=1} F^+(s; a/q)
& = \prod_{p\nmid q}(1+(f(p)-1)p^{-1})
\sum_{d\mid q} \frac{\mu(q/d)}{\varphi(q/d)} \frac{f(d)}{d}\prod_{p\mid \frac{q}{d}}(1-p^{-1})
\\ & = \prod_{p\nmid q}(1+(f(p)-1)p^{-1})
\frac{1}{q}\sum_{d\mid q} \mu\!\left(\frac{q}{d}\right) f(d)
\\ & = \frac{1}{q} \prod_{p\nmid q}(1+(f(p)-1)p^{-1})
\begin{cases}
\prod_{p\mid q}(f(p)-1) & \text{ if $q$ is squarefree,}\\
0 & \text{otherwise.}
\end{cases}
\end{align*}
When $q$ is squarefree, we get
\[
\Res_{s=1} F^+(s; a/q)
= \frac{C'}{q} \prod_{p\mid q}
(1+(f(p)-1)p^{-1})^{-1} (f(p)-1)
= C' \prod_{p\mid q}\frac{1}{(p^2-1)(p-1)}, 
\]
where
\[
C' = \prod_{p} \bigg(1+\frac{1}{p(p^2-p-1)}\bigg)
= \prod_p \bigg(\frac{(p^2-1)(p-1)}{p(p^2-p-1)}\bigg)
= \frac{1}{C}.
\]

At $s=-1$, using that $L(s,\chi_{q/d}) = \zeta(s) \prod_{p\mid \frac{q}{d}}(1-p^{-s})$, we have
\begin{align*}
\Res_{s=-1} F^+(s; a/q) 
&= \prod_{p\nmid q} \left(1+\frac{p^{-1}\frac{(p+1)}{p^2-p-1}}{1+p^{-1}}\right) 
\sum_{d\mid q} \frac{f(d)d}{\varphi(q/d)} 
\tau(\chi_{q/d}) L(-1, \chi_{q/d}) 
\frac{\prod_{p\mid q}(1-p^{-1})}{L(2, \chi_0)}
\\ &= \frac{\zeta(-1)}{\zeta(2)}\prod_{p\nmid q} f(p)
\sum_{d\mid q} \frac{f(d)d\mu(q/d)}{\varphi(q/d)} 
\prod_{p\mid\frac{q}{d}} (1-p)
\frac{\prod_{p\mid q}(1-p^{-1})}{\prod_{p\mid q}(1-p^{-2})}
\\ &= \frac{\zeta(-1)}{\zeta(2)}
\prod_{p\nmid q} f(p) 
\prod_{p\mid q} (1+p^{-1})^{-1}
\sum_{d\mid q} \frac{f(d) d\mu(q/d)^2}{\varphi(q/d)} 
\frac{q}{d} 
\prod_{p\mid\frac{q}{d}} (1-p^{-1})
\\ &= \frac{\zeta(-1)}{\zeta(2)} \prod_{p\nmid q}f(p) \prod_{p\mid q} (1+p^{-1})^{-1} 
\sum_{d\mid q}f(d) d\mu(q/d)^2
\\ &= \frac{\zeta(-1)}{\zeta(2)} \prod_{p\nmid q}f(p) \prod_{p\mid q} (1+p^{-1})^{-1} 
\prod_{p\parallel q} (1+f(p)p)
\prod_{p^2\mid q}f(p)(p^{\ord_p(q)} + p^{\ord_p(q)-1})
\\ &= \frac{\zeta(-1)}{\zeta(2)} \prod_p f(p)
\cdot q
\prod_{p\mid q}(1+p^{-1})^{-1} 
\prod_{p\parallel q} \frac{1+f(p)p}{pf(p)} \prod_{p^2\mid q} (1+p^{-1}).
\end{align*}
Since $\prod_p f(p) = f(0)$, $Cf(0) = \zeta(2)$ and $\zeta(-1)=-\frac{1}{12}$, we get
\[
\Res_{s=-1} F^+(s; a/q)
= -\frac{q}{12C}\prod_{p\parallel q}\frac{1+f(p)p}{f(p)(p+1)}.
\]

Next, when $\chi$ is an odd character, $L(s,\chi)$ has a trivial zero at $s=-1$, so from \eqref{e:Fsa/q_sumofLchi} we see that $F^-(-1;a/q)=0$.

Similarly, for even $\chi$, $L(s,\chi)$ has a trivial zero at $s=0$ unless $\chi$ is the character of modulus $1$. Thus the only nonzero term in \eqref{e:Fpm_first_expression} comes from $d=q$, so we have
\[
F^+(0;a/q)=f(q)\zeta(0)\prod_{p\nmid q}f(p)=-\frac12f(0).
\]
\end{proof}

\begin{lemma}\label{lem:tildeg}
Define 
\[
g^\pm(t;\theta,x) = \frac{e(\theta t)\pm e(-\theta t)}{2} \widehat{W}\!\left(\frac{t}{x}\right)
\quad\text{and}\quad
\tilde{g}^\pm(s; \theta,x) = \int_0^\infty g^\pm(t; \theta,x)t^s \, \frac{dt}{t}. 
\]
Then $\tilde{g}^\pm(s;\theta,x)$ has meromorphic continuation to $\C$, with at most simple poles in $\{s\in\Z:s\le 0,\,(-1)^s=\pm1\}$, and satisfies
\[
\tilde{g}^\pm(s;\theta,x)\ll_j
\frac{x^{\Re(s)}\max(1,x|\theta|)^j}{(\Re(s)+j)|s(s+1)\cdots(s+j-1)|}
\quad\text{for }j\in\Z_{\ge0}\text{ and }\Re(s)>-j.
\]
Furthermore, the following identities hold:
\begin{itemize}
\item $\tilde{g}^+(1;\theta,x) = \frac12xW(x\theta)$;
\item $\Res_{s=0}\tilde{g}^+(s;\theta,x) = 1$;
\item $\tilde{g}^+(-1;\theta,x) = \frac{\pi^2}{x}\int_\R\max(|v|,x|\theta|)W(v)\,dv$.
\end{itemize}
\end{lemma}
\begin{proof}
For brevity we assume that $\theta$ and $x$ are fixed throughout this proof and suppress them from the notation.
Applying integration by parts, for $j\ge0$ and $\Re(s)>-j$, we have
\[
\tilde{g}^\pm(s)=\frac{(-1)^j}{s(s+1)\cdots(s+j-1)}\int_0^\infty
(g^\pm)^{(j)}(t)t^{s+j-1}\,dt.
\]
Note that when $j$ is odd, $(g^+)^{(j)}(0)=0$, so this integral representation is valid for $\Re(s)>-j-1$, and it follows shows that $\tilde{g}^+(s)$ is analytic at $s=-j$. Similarly, $\tilde{g}^-(s)$ is analytic at $s=-j$ for even $j$.

By the Leibniz rule, for any $j\ge0$,
\[
\big|(g^\pm)^{(j)}(t)\big|\le\sum_{b=0}^j{j\choose b}|2\pi\theta|^b
x^{b-j}\left|\widehat{W}^{(j-b)}\!\left(\frac{t}{x}\right)\right|
\ll_jx^{-j}\max(1,x|\theta|)^j.
\]
Inserting this estimate into the integral representation above, for $\Re(s)>-j$ we have
\[
\tilde{g}^\pm(s)\ll_j
\frac{x^{-j}\max(1,x|\theta|)^j}{|s(s+1)\cdots(s+j-1)|}
\int_0^xt^{\Re(s)+j-1}\,dt
=\frac{x^{\Re(s)}\max(1,x|\theta|)^j}{(\Re(s)+j)|s(s+1)\cdots(s+j-1)|}.
\]

Turning to the additional identities, we first have
\[
\tilde{g}^+(1)=\int_0^\infty\cos(2\pi\theta t)\widehat{W}\!\left(\frac{t}{x}\right)dt
=\frac{x}{2}\int_\R e(x\theta t)\widehat{W}(t)\,dt
=\frac{x}{2}W(x\theta).
\]
Second, using the integral representation with $j=1$, we have
\[
\Res_{s=0}\tilde{g}^+(s)=-\int_0^\infty(g^+)'(t)\,dt
=g^+(0)=1.
\]

Third, for $u\in\R$, let 
\[G(u) = \int_{-\infty}^\infty (g^+)'(t) t^{-1} e(-ut) \, dt.\]
Then 
\[G'(u) = -2\pi i \int_{-\infty}^\infty (g^+)'(t) e(-ut) \, dt
= -4\pi^2 u \int_{-\infty}^\infty g^+(t)  e(-ut) \, dt 
= -4\pi^2 u \widehat{g^+}(u).\]
Since $G(u)\to0$ as $|u|\to \infty$, we have
\[G(u) = \int_{-\infty}^u G'(v) \, dv 
= -4\pi^2 \int_{-\infty}^u v \widehat{g^+}(v) \, dv.\]
Taking $u=0$, we get
\[\tilde{g}^+(-1) 
= \frac12\int_{-\infty}^\infty (g^+)'(t)t^{-1}\,dt
= -2\pi^2\int_{-\infty}^0 v\widehat{g^+}(v)\, dv
= 2\pi^2\int_0^\infty v\widehat{g^+}(v)\,dv.\]

We now compute $\widehat{g^+}(v)$. 
Recalling the definition of $g^+$, we have
\begin{align*} 
\widehat{g^+}(v) 
& = \int_{-\infty}^\infty g^+(u) e(-uv) \, du
\\ & = \frac{1}{2}\bigg\{\int_{-\infty}^\infty \widehat{W}\!\left(\frac{u}{x}\right) e(-u(v-\theta)) \, du
+ \int_{-\infty}^\infty \widehat{W}\!\left(\frac{u}{x}\right) e(-u(v+\theta)) \, du
\bigg\}
\\ & = \frac{x}{2} \bigg\{W(x(v-\theta)) + W(x(v+\theta))
\bigg\}.
\end{align*}
Therefore
\begin{align*}
\tilde{g}^+(-1)
& = \pi^2x\int_0^\infty v \bigg\{W(x(v-\theta)) + W(x(v+\theta))\bigg\} \, dv
\\ & = \pi^2 \bigg\{\frac{1}{x} \int_{-x|\theta|}^\infty v W(v)\, dv 
+ \frac{1}{x} \int_{x|\theta|}^\infty v W(v) \, dv
+ \int_{-x|\theta|}^{x|\theta|}|\theta| W(v) \, dv \bigg\}
\\ & = 2\pi^2 \bigg\{\frac{1}{x} \int_{x|\theta|}^\infty v W(v) \, dv
+ \int_0^{x|\theta|} |\theta| W(v) \, dv \bigg\}
\\ &= 2\pi^2\int_0^\infty\max\!\left(\frac{v}{x},|\theta|\right)W(v)\,dv.
\end{align*}
\end{proof}

We can now complete the proof of Proposition~\ref{prop:circle}.
Recall again from Lemma~\ref{lem:L1_def} that $L(1, \widetilde{\psi}_t) = C f(t)$.
In the notation of Lemma~\ref{lem:tildeg}, we have
\begin{align*}
\sum_{t\in\Z}\cos(2\pi\alpha t) f(t) \widehat{W}\!\left(\frac{t}{x}\right)
&= f(0)
+ 2\sum_{t=1}^\infty f(t) \cos(2\pi\alpha t) \widehat{W}\!\left(\frac{t}{x}\right)\\
&= f(0)
+2\sum_{t=1}^\infty f(t)\cos\!\left(\frac{2\pi at}{q}\right)g^+(t;\theta,x)
+2i\sum_{t=1}^\infty\sin\!\left(\frac{2\pi at}{q}\right)g^-(t;\theta,x).
\end{align*}
By Mellin inversion, we have 
\[
g^\pm(t;\theta,x) = \frac{1}{2\pi i}\int_{(2)} \tilde{g}^\pm(s;\theta,x) t^{-s} \, ds,
\]
so that
\[
\sum_{t=1}^\infty f(t)\cos\!\left(\frac{2\pi at}{q}\right)g^+(t;\theta,x)
= \frac{1}{2\pi i}\int_{(2)} \tilde{g}^+(s;\theta,x)
F^+(s; a/q)\, ds.
\]
Shifting the contour to $\Re(s)=\sigma$ for $-\frac{7}{4} < \sigma \le -\frac32$, we get
\begin{multline*}
\tilde{g}^+(1;\theta,x)\Res_{s=1}F^+(s;a/q)
+F^+(0;a/q)\Res_{s=0}\tilde{g}^+(s;\theta,x)
\\ +\tilde{g}^+(-1;\theta,x)\Res_{s=-1}F(s;a/q)
+ \frac{1}{2\pi i} \int_{(\sigma)} \tilde{g}^+(s;\theta,x) F^+(s; a/q) \, ds
\\ =\frac{1}{C}\frac{\mu(q)^2}{\varphi(q)^2\sigma(q)}\frac{x}{2}W(x\theta) 
-\frac12f(0)
-\frac12f(0)\frac{q}{x}\prod_{p\parallel q}\left(1-\frac1{p^3-p}\right)\cdot
\int_\R\max(|v|,x|\theta|)W(v)\,dv
\\ + \frac{1}{2\pi i} \int_{(\sigma)}\tilde{g}^+(s;\theta,x) F(s; a/q) \, ds. 
\end{multline*}
Similarly we have 
\begin{align*}
i\sum_{t=1}^\infty f(t)\sin\!\left(\frac{2\pi at}{q}\right)g^-(t;\theta,x)
&= \frac{1}{2\pi i}\int_{(2)} \tilde{g}^-(s;\theta,x)
F^-(s; a/q)\, ds\\
&=  \frac{1}{2\pi i}\int_{(\sigma)} \tilde{g}^-(s;\theta,x)
F^-(s; a/q)\, ds.
\end{align*}
Note that there is no residue term in this case since $F^-(-1; a/q)=0$.

Therefore, for $-\frac74 < \sigma \le -\frac32$, we have
\begin{align*}
\sum_{t\in\Z} \cos(2\pi\alpha t) f(t) \widehat{W}\!\left(\frac{t}{x}\right)
&= \frac{1}{C} \frac{\mu(q)^2}{\varphi(q)^2\sigma(q)}xW(x\theta) \nonumber \\
&- f(0)\frac{q}{x}\prod_{p\parallel q}\left(1-\frac1{p^3-p}\right)\cdot
\int_\R\max(|v|,x|\theta|)W(v)\,dv\\
&+ \frac{1}{\pi i} \int_{(\sigma)} \left(\tilde{g}^+(s;\theta,x) F^+(s; a/q) + \tilde{g}^-(s;\theta,x) F^-(s; a/q)\right)\, ds.
\end{align*}
The second line is $O\bigl(qx^{-1}\max(1,x|\theta|)\bigr)$. As for the third line, assuming GRH and applying Lemma~\ref{lem:F}, for $\Re(s)=\sigma\in\bigl(-\frac74,-\frac32\bigr]$ we have
\begin{align*}
F^\pm(s; a/q)
& \ll_{\sigma,\varepsilon}\sum_{d\mid q}\frac{1}{\varphi(q/d)}\frac{f(d)}{d^\sigma}
\sum_{\substack{\chi\bmod{\frac{q}{d}}\\\chi(-1)=\pm1}} |\tau(\overline\chi)|
\left|\frac{q}{d}s\right|^{\frac12-\sigma+\frac12-(\sigma+2)+\varepsilon}
\\ & \ll
|s|^{-1-2\sigma+\varepsilon} \sum_{d\mid q} \frac{q^{-\frac12-2\sigma+\varepsilon}}{d^{-\frac12-\sigma+\varepsilon}}
\ll_\varepsilon|s|^{-1-2\sigma+\varepsilon}q^{-\frac12-2\sigma+\varepsilon}.
\end{align*}
So we get
\[
\frac{1}{\pi i} \int_{(\sigma)} \left(\tilde{g}^+(s) F^+(s; a/q) + \tilde{g}^-(s) F^-(s; a/q)\right)\, ds
\ll_{\sigma,\varepsilon} \sum_\pm q^{-\frac12-2\sigma+\varepsilon}\int_{(\sigma)} |\tilde{g}^\pm(s)| |s|^{-1-2\sigma+\varepsilon}\, |ds|.
\]
By Lemma~\ref{lem:tildeg}, for $\Re(s)=\sigma\in\bigl(-\frac74,-\frac32\bigr]$ and $j\ge2$,
\[
\tilde{g}^\pm(s)\ll_j
x^\sigma\frac{\max(1,x|\theta|)^j}{(\sigma+j)|s(s+1)\cdots(s+j-1)|}
\ll_j x^\sigma|s|^{-j}\max(1,x|\theta|)^j.
\]
We use $j=2$ for $|s|<\max(1,x|\theta|)$ and $j=4$ for $|s|\ge\max(1,x|\theta|)$, obtaining
\[
\ll_{\sigma,\varepsilon}
q^{-\frac12-2\sigma+\varepsilon}x^\sigma\max(1,x|\theta|)^{-2\sigma+\varepsilon}.
\]
Choosing $\sigma=-\frac74+\frac{\varepsilon}{2}$ gives
\[
\ll_\varepsilon q^3x^{-\frac74+\varepsilon}\max(1,x|\theta|)^{\frac72}.
\]

\subsection{The Fourier series}\label{ssec:fourier}
We conclude this section by proving equality of the two expressions for $\nu(E)$ appearing in Theorem~\ref{thm:main}. Formally this should follow from Proposition~\ref{prop:circle}, but we give a direct proof here.

Consider the function $S:\R_{>0}\to\R$ defined by
\[
S(\alpha)=\frac12-\zeta(2)\alpha+\sum_{\frac{a}{q}\in\Q\cap(0,\alpha]}\frac{\mu(q)^2}{\varphi(q)^2\sigma(q)}.
\]
(If $\alpha\in\Q$ then the corresponding term of the sum is counted with the full weight.) For any $\alpha$, we have
\begin{align*}
S(\alpha+1)-S(\alpha)&=-\zeta(2)+\sum_{q=1}^\infty\sum_{\substack{\alpha q<a\le\alpha q+q\\(a,q)=1}}\frac{\mu(q)^2}{\varphi(q)^2\sigma(q)}\\
&=-\zeta(2)+\sum_{q=1}^\infty\frac{\mu(q)^2}{\varphi(q)\sigma(q)}
\\
&=-\zeta(2)+\prod_p\left(1+\frac1{(p-1)(p+1)}\right)=0,
\end{align*}
so $S(\alpha)$ is periodic mod $1$ and has bounded variation.

It follows that the Fourier expansion of $S(\alpha)$ converges at every point. We proceed to calculate its Fourier coefficients:
\begin{align*}
\int_0^1 S(\alpha)e(-\alpha t)\,d\alpha
&=\int_0^1\Bigg(\frac12-\zeta(2)\alpha+\sum_{\frac{a}{q}\in\Q\cap(0,\alpha]}\frac{\mu(q)^2}{\varphi(q)^2\sigma(q)}\Bigg)e(-\alpha t)\,d\alpha\\
&=\int_0^1(\tfrac12-\zeta(2)\alpha)e(-\alpha t)\,d\alpha+\sum_{q=1}^\infty\sum_{\substack{1\le a\le q\\(a,q)=1}}
\frac{\mu(q)^2}{\varphi(q)^2\sigma(q)}\int_{a/q}^1e(-\alpha t)\,d\alpha.
\end{align*}
When $t=0$ this is
\[
\frac{1-\zeta(2)}2+\sum_{q=1}^\infty\frac{\mu(q)^2}{\varphi(q)^2\sigma(q)}\sum_{\substack{1\le a\le q\\(a,q)=1}}\left(1-\frac{a}{q}\right).
\]
Since $\gcd(q-a,q)=\gcd(a,q)$, the inner sum is
\[
\sum_{\substack{1\le a\le q\\(a,q)=1}}\left(1-\frac{a}{q}\right)
=\begin{cases}
0&\text{if }q=1,\\
\frac{\varphi(q)}{2}&\text{if }q>1,
\end{cases}
\]
so we get
\[
\frac{1-\zeta(2)}{2}+\frac12\sum_{q=2}^\infty\frac{\mu(q)^2}{\varphi(q)\sigma(q)}=0.
\]

For $t\ne0$ we get
\begin{align*}
\frac{\zeta(2)}{2\pi it}
+\sum_{q=1}^\infty\sum_{\substack{1\le a\le q\\(a,q)=1}}
\frac{\mu(q)^2}{\varphi(q)^2\sigma(q)}
\frac{e(-at/q)-1}{2\pi it}
&=\frac{\zeta(2)}{2\pi it}
+\sum_{q=1}^\infty
\frac{\mu(q)^2}{\varphi(q)^2\sigma(q)}
\frac{c_q(-t)-\varphi(q)}{2\pi it}\\
&=\frac1{2\pi it}\sum_{q=1}^\infty
\frac{\mu(q)^2c_q(t)}{\varphi(q)^2\sigma(q)},
\end{align*}
where $c_q(t)=\sum_{a\in(\Z/q\Z)^\times}e(at/q)$ is the Ramanujan sum. By \cite[\S4.1, Theorem~1]{MV}, we have
\[
c_q(t) = \mu\!\left(\frac{q}{(q,t)}\right)\frac{\varphi(q)}{\varphi\bigl(\frac{q}{(q,t)}\bigr)},
\]
so this becomes
\begin{align*}
\frac1{2\pi it}\sum_{q=1}^\infty
\frac{\mu(q)\mu((q,t))\varphi((q,t))}{\varphi(q)^2\sigma(q)}
&=\frac1{2\pi it}\prod_{p\nmid t}
\left(1-\frac1{(p-1)^2(p+1)}\right)
\prod_{p\mid t}\left(1+\frac1{(p-1)(p+1)}\right)\\
&=\frac{L(1,\widetilde{\psi}_t)}{2\pi it}.
\end{align*}

Therefore, $S(\alpha)$ has Fourier series
\[
\sum_{t=1}^\infty\frac{L(1,\widetilde{\psi}_t)}{\pi t}\sin(2\pi\alpha t).
\]
By the Dirichlet--Dini criterion, at the jump discontinuities of $S(\alpha)$ (i.e.\ at every rational with squarefree denominator), the series converges to the average of the left and right limits. Thus, it equals
\[
S^\ast(\alpha):=-\zeta(2)\alpha+\sideset{}{^\ast}\sum_{\frac{a}{q}\in\Q\cap[0,\alpha]}\frac{\mu(q)^2}{\varphi(q)^2\sigma(q)},
\]
where the $\ast$ indicates that the endpoints are weighted by $\frac12$, as in \eqref{e:nu_def}.

Next, to relate this to $\nu$, we define a measure $\lambda$ on $\R_{>0}$ that assigns mass $\frac1{\zeta(2)}\frac{\mu(q)^2}{\varphi(q)^2\sigma(q)}\left(\frac{q}{a}\right)^3$ to each $\frac{a}{q}\in\Q_{>0}$.
By Stieltjes integration, for $\alpha_2>\alpha_1>0$, we have
\begin{align*}
\lambda((\alpha_1,\alpha_2])&=\frac1{\zeta(2)}\int_{\alpha_1}^{\alpha_2}
\alpha^{-3}\,d(S(\alpha)+\zeta(2)\alpha)
=\int_{\alpha_1}^{\alpha_2}\alpha^{-3}\,d\alpha
+\frac1{\zeta(2)}\int_{\alpha_1}^{\alpha_2}\alpha^{-3}\,dS(\alpha)\\
&=\int_{\alpha_1}^{\alpha_2}\alpha^{-3}\,d\alpha
+\frac{S(\alpha)}{\zeta(2)\alpha^3}\bigg|_{\alpha_1}^{\alpha_2}
+\frac3{\zeta(2)}\int_{\alpha_1}^{\alpha_2}\alpha^{-4}S(\alpha)\,d\alpha.
\end{align*}
We can replace $S(\alpha)$ by $S^\ast(\alpha)$ in the integral, since they differ on a set of measure $0$:
\begin{align*}
\frac3{\zeta(2)}\int_{\alpha_1}^{\alpha_2}\alpha^{-4}S(\alpha)\,d\alpha
&=\frac3{\zeta(2)}\int_{\alpha_1}^{\alpha_2}\alpha^{-4}S^\ast(\alpha)\,d\alpha\\
&=\frac3{\zeta(2)}\int_{\alpha_1}^{\alpha_2}\alpha^{-4}
\sum_{t=1}^\infty\frac{L(1,\widetilde{\psi}_t)}{\pi t}\sin(2\pi\alpha t)\,d\alpha.
\end{align*}
Since $S^\ast(\alpha)$ is square-integrable, we are free to change the order of sum and integral:
\begin{align*}
\frac3{\zeta(2)}\int_{\alpha_1}^{\alpha_2}\alpha^{-4}S(\alpha)\,d\alpha&=
\frac3{\zeta(2)}\sum_{t=1}^\infty\frac{L(1,\widetilde{\psi}_t)}{\pi t}
\int_{\alpha_1}^{\alpha_2}\sin(2\pi\alpha t)\alpha^{-4}\,d\alpha\\
&=\frac1{\zeta(2)}\sum_{t=1}^\infty\frac{L(1,\widetilde{\psi}_t)}{\pi t}
\left(-\frac{\sin(2\pi\alpha t)}{\alpha^3}\bigg|_{\alpha_1}^{\alpha_2}+
2\pi t\int_{\alpha_1}^{\alpha_2}\cos(2\pi\alpha t)\alpha^{-3}\,d\alpha\right)\\
&=-\frac{S^\ast(\alpha)}{\zeta(2)\alpha^3}\bigg|_{\alpha_1}^{\alpha_2}
+\frac2{\zeta(2)}\sum_{t=1}^\infty L(1,\widetilde{\psi}_t)\int_{\alpha_1}^{\alpha_2}
\cos(2\pi\alpha t)\alpha^{-3}\,d\alpha.
\end{align*}
Substituting back into the above and using that $L(1,\widetilde{\psi}_t)=\zeta(2)\prod_{p\nmid t}\frac{p^2-p-1}{p^2-p}$, we get
\begin{align*}
\lambda((\alpha_1,\alpha_2])
&=\frac{S(\alpha)-S^\ast(\alpha)}{\zeta(2)\alpha^3}\bigg|_{\alpha_1}^{\alpha_2}
+\sum_{t=-\infty}^\infty\frac{L(1,\widetilde{\psi}_t)}{\zeta(2)}\int_{\alpha_1}^{\alpha_2}\cos(2\pi\alpha t)\alpha^{-3}\,d\alpha\\
&=\lambda((\alpha_1,\alpha_2])-\nu\bigl(\bigl[\alpha_2^{-2},\alpha_1^{-2}\bigr]\bigr)
+\frac12\sum_{t=-\infty}^\infty\prod_{p\nmid t}\frac{p^2-p-1}{p^2-p}\cdot \int_{\alpha_2^{-2}}^{\alpha_1^{-2}}\cos\!\left(\frac{2\pi t}{\sqrt{y}}\right)dy.
\end{align*}
This completes the proof.

\section{The circle method}\label{sec:circle}
Recall that $\lambda_{k_0}=\frac{k_0-1}{4\pi \sqrt{N}}$ and $x_{k_0}(\alpha)=\frac{k_0-1}{4\alpha h}$.
We further define $\Delta=\frac{H}{4\pi\sqrt{N}}$,
$P=\frac1{3\sqrt{\Delta\alpha_2}}$, and 
$Q=\frac{H}{4\Delta\alpha_2hP}$. Note that $P\asymp_E\sqrt{\frac{K}{H}}$, $Q\asymp_E\frac{\sqrt{HK}}{h}$ and $x_{k_0}(\alpha)\asymp_E PQ\asymp_E\frac{K}{h}$.
For $a,q\in\Z$ with $q>0$ and $(a,q)=1$, define
\[
\M\!\left(\frac{a}{q}\right)=\left\{\frac{a}{q}+\theta:|\theta|\le\frac1{qQ}\right\}.
\]
Since $4h<H$, we have
\[\frac{P}{Q} =
\frac{4h}{9H}<\frac19<\frac12,
\]
and it follows that $\M(\frac{a}{q})\cap\M(\frac{a'}{q'})=\emptyset$ when $0<q,q'\le P$ and $\frac{a}{q}\ne \frac{a'}{q'}$.

For a fixed $k_0$, we decompose the interval $I_{k_0}:=[\lambda_{k_0}\alpha_1,\lambda_{k_0}\alpha_2]$ into major arcs $\M_{k_0}$ and minor arcs $\m_{k_0}$, defined by
\[
\M_{k_0}=I_{k_0}\cap \bigcup_{\substack{a,q\in\Z\\0<q\le P\\(a,q)=1}}\M\!\left(\frac{a}{q}\right)
\quad\text{and}\quad
\m_{k_0}=I_{k_0}\setminus\M_{k_0}.
\]
Note that the definitions of $P$, $Q$, and $\M(\frac{a}{q})$ are independent of $k_0$; only their intersection with $I_{k_0}$ can vary. As we will show, near each endpoint there is at most one fraction $\frac{a}{q}$ with $0<q\le P$ for which $\M(\frac{a}{q})\cap I_{k_0}$ varies.

To that end, since $4\pi\sqrt{N}=K-1+O(1/K)$, we have
\[
\big|\lambda_{k_0}-1\big| = \frac{|k_0-K|+O(1/K)}{4\pi\sqrt{N}}\le\frac{H-4h+O(1/K)}{4\pi\sqrt{N}}\le\Delta
\]
for sufficiently large $K$. Thus, as $k_0$ varies, the endpoint $\lambda_{k_0}\alpha_i$ is confined to the interval $\bigl[(1-\Delta)\alpha_i,(1+\Delta)\alpha_i\bigr]$.
Next, by Dirichlet's theorem, for $i=1,2$ we can choose a fraction $\frac{a_i}{q_i}$ such that $q_i\le 3P$ and
\[
\alpha_i=\frac{a_i}{q_i}+\theta_i,
\quad\text{with }
|\theta_i|\le\frac1{3q_iP}.
\]
Assume that $K$ is sufficiently large to ensure that $a_i>0$.

Let $\frac{a}{q}$ be a fraction with $0<q\le P$ and $\frac{a}{q}\notin\{\frac{a_1}{q_1},\frac{a_2}{q_2}\}$. Adding the inequalities
\[
qq_i\Delta\alpha_i\le3P^2\Delta\alpha_2=\frac13,
\quad
\frac{q_i}{Q}\le\frac{3P}{Q}<\frac13,
\quad\text{and}\quad
qq_i|\theta_i|\le\frac{q}{3P}\le\frac13,
\]
we have
\begin{align*}
qq_i|\theta_i| + \frac{q_i}{Q} + qq_i\Delta\alpha_i<1
&\implies
|\theta_i| + \frac1{qQ} + \Delta\alpha_i <\frac1{qq_i}\le\left|\frac{a}{q}-\frac{a_i}{q_i}\right|\\
&\implies
\left|\frac{a}{q}-\alpha_i\right|
>\frac1{qQ}+\Delta\alpha_i.
\end{align*}
Thus,
\[
\M\!\left(\frac{a}{q}\right)\cap
\bigl[\alpha_i(1-\Delta),\alpha_i(1+\Delta)\bigr]=\emptyset.
\]
Therefore, $\M(\frac{a}{q})\cap I_{k_0}$ does not depend on $k_0$; in fact, recalling that $E=\bigl[\alpha_2^{-2},\alpha_1^{-2}\bigr]$, we have
\[
\M\!\left(\frac{a}{q}\right)\cap I_{k_0}=\begin{cases}
\M(\frac{a}{q})&\text{if }(\frac{a}{q})^{-2}\in E,\\
\emptyset&\text{if }(\frac{a}{q})^{-2}\notin E.
\end{cases}
\]

We split the integral over $I_{k_0}$ in \eqref{e:psum_result}
as
\begin{multline*}
\int_{I_{k_0}}\sum_{t\in\Z}L(1,\widetilde{\psi}_t)\cos(2\pi\alpha t) \widehat{W}\!\left(\frac{t}{x_{k_0}(\alpha)}\right)\frac{d\alpha}{\alpha^3}\\
=\int_{\m_{k_0}}+\int_{\M_{k_0}}
=\int_{\m_{k_0}}+\sum_{\substack{\frac{a}{q}\in(\Q\setminus\{\frac{a_1}{q_1},\frac{a_2}{q_2}\})\cap[\alpha_1,\alpha_2]\\q\le P}}\int_{\M(\frac{a}{q})}
+\sum_{i=1}^2\int_{\M_{k_0}\cap\M(\frac{a_i}{q_i})}.
\end{multline*}
Note that the integral over $\M_{k_0}\cap\M(\frac{a_i}{q_i})$ vanishes if $q_i>P$.

We evaluate the terms of this sum using Proposition~\ref{prop:circle}. By Dirichlet's theorem, for $\alpha\in I_{k_0}$ we may choose $q\le Q$ and $a$ coprime to $q$ such that
\[
\left|\alpha-\frac{a}{q}\right|\le\frac1{qQ}.
\]
If $\alpha\in\m_{k_0}$ then $q>P$, in which case
$x_{k_0}(\alpha)|\alpha-\frac{a}{q}|\ll_E1$, so by Proposition~\ref{prop:circle},
\begin{align*}
\sum_{t\in\Z}L(1,\widetilde{\psi}_t) \cos(2\pi\alpha t) \widehat{W}\!\left(\frac{t}{x_{k_0}(\alpha)}\right)
&\ll_{E,\varepsilon}
\frac{x_{k_0}(\alpha)}{P^{3-\varepsilon}}
+ Qx_{k_0}(\alpha)^{-1} +
Q^3x_{k_0}(\alpha)^{-\frac74+\varepsilon}\\
&\ll_{E,\varepsilon}h^{-1}H^{\frac32}K^{-\frac12+\varepsilon}+H^{\frac12}K^{-\frac12}+h^{-\frac54}H^{\frac32}K^{-\frac14+\varepsilon}\\
&\ll h^{-\frac54}H^{\frac32}K^{-\frac14+\varepsilon}.
\end{align*}
Note that this estimate also applies to the error terms in Proposition~\ref{prop:circle} for $\alpha$ in a major arc $\M(\frac{a}{q})$ with $|\alpha-\frac{a}{q}|x_{k_0}(\alpha)\le1$. For $\alpha\in\M(\frac{a}{q})$ with $|\alpha-\frac{a}{q}|x_{k_0}(\alpha)>1$, the error term is
\[
\ll_\varepsilon Q^{-1} + q^{-\frac12}Q^{-\frac72}x_{k_0}(\alpha)^{\frac74+\varepsilon}
\ll_{E,\varepsilon}Q^{-1} + q^{-\frac12}Q^{-\frac72}(K/h)^{\frac74+\varepsilon}.
\]
Summing this error term times the measure of each major arc gives
\[
\ll_{E,\varepsilon}PQ^{-2} + P^{\frac12}Q^{-\frac92}(K/h)^{\frac74+\varepsilon}\ll_{E,\varepsilon} h^2H^{-\frac32}K^{-\frac12} + h^{\frac{11}{4}}H^{-\frac52}K^{-\frac14+\varepsilon},
\]
which is dominated by the minor arc error.

Multiplying by $\frac{2h}{\pi}\left(\frac{k_0-1}{4\pi}\right)^2$, summing over the $J\ll H/h$ choices of $k_0$, combining with the other error terms in \eqref{e:psum_result}, and using
$h\asymp\max\bigl(H^{\frac{10}{9}}K^{-\frac19},(HK)^{\frac{5}{11}}\bigr)$, we get
\begin{equation}\label{minor_arc_error}
\ll_{E,\varepsilon}K^\varepsilon\Bigl(hK^2+h^{-\frac65}HK^3+h^{-\frac54}H^{\frac52}K^{\frac74}\Bigr)
\ll HK^{2+\varepsilon}\Bigl(H^{\frac19}K^{-\frac19}+
H^{-\frac{6}{11}}K^{\frac{5}{11}}\Bigr).
\end{equation}

We turn now to the main term in Proposition~\ref{prop:circle} on the major arcs.
Note that for $\alpha\in I_{k_0}$, we have
$x_{k_0}(\alpha)\ge x_{k_0}(\lambda_{k_0}\alpha_2)
=PQ$.
Therefore, when $(\frac{a}{q})^{-2}\in E$ and $\frac{a}{q}\notin\{\frac{a_1}{q_1},\frac{a_2}{q_2}\}$, we have
\[
\left\{x_{k_0}(\alpha)\left(\alpha-\frac{a}{q}\right):\alpha\in\M\!\left(\frac{a}{q}\right)\right\}
\supseteq[-1,1].
\]
Hence, making the change of variables $x=x_{k_0}(\alpha)(\alpha-\frac{a}{q})$, we obtain 
\begin{align*}
\int_{\M(\frac{a}{q})}
x_{k_0}(\alpha)W\!\left(x_{k_0}(\alpha)\left(\alpha-\frac{a}{q}\right)\right)\frac{d\alpha}{\alpha^3}
&=\left(\frac{a}{q}\right)^{-3} 
\int_{\R} W(x)\left(1-\frac{4hx}{k_0-1}\right)^2\,dx\\
&=\left(\frac{a}{q}\right)^{-3} 
\left(1+\left(\frac{4h}{k_0-1}\right)^2
\int_\R x^2W(x)\,dx\right).
\end{align*}
Multiplying by $\frac{2h}{\pi}\left(\frac{k_0-1}{4\pi}\right)^2$
and summing over $k_0=K_0+4hj$, we get
\[
\left(\frac{a}{q}\right)^{-3} \frac{2h^3J}{\pi^3}
\bigg\{
\bigg(\frac{J-1}{2}+\frac{K_0-1}{4h}\bigg)^2
+ \frac{J^2-1}{12}
+\int_\R x^2W(x)\,dx
\bigg\}.
\]
Since $J=\left\lfloor\frac{K+H-K_0}{4h}\right\rfloor=\frac{H}{2h}+O(1)$ and
$\frac{J-1}{2}+\frac{K_0-1}{4h} 
= \frac{K}{4h}+O(1)$,
this is
\begin{multline*}
\left(\frac{a}{q}\right)^{-3} \frac{1}{\pi^3} 
\bigg(H+O(2h)\bigg)
\bigg\{
\frac{K^2}{16} + \frac{H^2}{48}
+O(hK)
+ O(h^2)
+h^2\int_\R x^2W(x)\,dx
\bigg\}
\\ = \left(\frac{a}{q}\right)^{-3} \frac{HK^2+ O(hK^2)}{16\pi^3}.
\end{multline*}

Multiplying by $\frac{\mu(q)^2}{\varphi(q)^2\sigma(q)^2}$ and summing over the major arcs (excluding $\frac{a_i}{q_i}$), we get
\[
\frac{HK^2}{16\pi^3}\sum_{\substack{\frac{a}{q}\in(\Q\setminus\{\frac{a_1}{q_1},\frac{a_2}{q_2}\})\cap[\alpha_1,\alpha_2]\\q\le P}}
\frac{\mu(q)^2}{\varphi(q)^2\sigma(q)^2}\left(\frac{a}{q}\right)^{-3}
+O_E\bigl(hK^2\bigr).
\]
Since $h\ll H^{\frac{10}{9}}K^{-\frac19}$ and
\[
\sum_{\substack{\frac{a}{q}\in\Q\cap[\alpha_1,\alpha_2]\\q>P}}
\frac{\mu(q)^2}{\varphi(q)^2\sigma(q)^2}\left(\frac{a}{q}\right)^{-3}
\ll_E\frac1P\ll_E\left(\frac{H}{K}\right)^{\frac12},
\]
we can write this as
\begin{equation}\label{major_arc_total}
\frac{HK^2}{16\pi^3}\sum_{\frac{a}{q}\in(\Q\setminus\{\frac{a_1}{q_1},\frac{a_2}{q_2}\})\cap[\alpha_1,\alpha_2]}
\frac{\mu(q)^2}{\varphi(q)^2\sigma(q)^2}\left(\frac{a}{q}\right)^{-3}
+O_E\!\left(H^{\frac{10}{9}}K^{\frac{17}{9}}\right).
\end{equation}

On RH, we have
\begin{equation}\label{denominator_total}
\frac1{\sqrt{N}}\sum_{\substack{p\,\mathrm{prime}\\p/N\in E}}\log{p}
\sum_{\substack{k\equiv2\delta\bmod4\\|k-K|\le H}}\sum_{f\in H_k(1)}1
=\frac{HK^2}{96\pi}|E| + O\bigl(K^2\bigr).
\end{equation}
Adding the error from \eqref{minor_arc_error} to \eqref{major_arc_total} and dividing by \eqref{denominator_total}
gives
\begin{equation}\label{almost_final_answer}
\frac1{\zeta(2)|E|}
\sum_{\frac{a}{q}\in(\Q\setminus\{\frac{a_1}{q_1},\frac{a_2}{q_2}\})\cap[\alpha_1,\alpha_2]}
\frac{\mu(q)^2}{\varphi(q)^2\sigma(q)^2}\left(\frac{a}{q}\right)^{-3}+
O_{E,\varepsilon}\!\left(K^\varepsilon\Bigl(H^{\frac19}K^{-\frac19}+
H^{-\frac{6}{11}}K^{\frac{5}{11}}\Bigr)\right),
\end{equation}
which is the expected answer, apart from a possible difference in the contributions from $\frac{a_i}{q_i}$.

When $\alpha_i$ is irrational, we bound the discrepancy trivially as
\[
\ll_E\frac{\mu(q_i)^2}{\varphi(q_i)^2\sigma(q_i)}
\ll_\varepsilon q_i^{-3+\varepsilon}.
\]
This tends to $0$ as $K\to\infty$, but we have no control over the rate of convergence in general. However, if $\alpha_i$ has irrationality measure $\mu_i\le27$ then this is dominated by the other error terms.

In more detail, suppose $(a_i/q_i)^{-2}\in E$ but the major arc integral does not always cover the full support of $W$, which happens when
\[
(-1)^i x_{k_0}(\lambda_{k_0}\alpha_i)\left(\lambda_{k_0}\alpha_i-\frac{a_i}{q_i}\right) < 1
\]
for some $k_0$. Rearranging this inequality, we find
\[
\left|\alpha_i-\frac{a_i}{q_i}\right|
< \alpha_i\left(\frac{h}{\pi\sqrt{N}}+|\lambda_{k_0}-1|\right)<2\Delta\alpha_i
\ll_E\frac{H}{K}.
\]
A similar calculation leads to the same inequality when $(a_i/q_i)^{-2}\notin E$ but a major arc overlaps the support of $W$ for some $k_0$.
On the other hand, by the definition of irrationality measure, we have $|\alpha_i-a_i/q_i|\ge q_i^{-\mu_i+o(1)}$ as $K\to\infty$, whence
\[
\frac{\mu(q_i)^2}{\varphi(q_i)^2\sigma(q_i)}
\le q_i^{-3+o(1)}\le\left(\frac{H}{K}\right)^{\frac3{\mu_i}-o(1)}.
\]
If $\mu_i\le27$ then this is dominated by the error term in \eqref{almost_final_answer}.

When $\alpha_i$ is rational, we have $\frac{a_i}{q_i}=\alpha_i$ for sufficiently large $K$.  Writing $x=x_{k_0}(\alpha)(\alpha-\alpha_i)$, for $\alpha\in I_{k_0}$ we have
\begin{align*}
(-1)^i(\lambda_{k_0}\alpha_i-\alpha)\ge0
&\implies
(-1)^i\left(\frac{(k_0-1)\alpha_i}{4h\alpha}
-\frac{k_0-1}{4h\lambda_{k_0}}\right)\ge0\\
&\implies
(-1)^i\left(-x+\frac{k_0-1}{4h}-\frac{k_0-1}{4h\lambda_{k_0}}\right)\ge0\\
&\implies
(-1)^ix\le
(-1)^i\frac{k_0-1-4\pi\sqrt{N}}{4h}.
\end{align*}
Thus,
\[
\int_{\M(\frac{a_1}{q_1})\cap I_{k_0}}
=\int_{\frac{k_0-1-4\pi\sqrt{N}}{4h}}^\infty
\quad\text{and}\quad
\int_{\M(\frac{a_2}{q_2})\cap I_{k_0}}
=\int_{-\infty}^{\frac{k_0-1-4\pi\sqrt{N}}{4h}}.
\]

Recall that $k_0=K_0+4hj$, where $j=0,\ldots,J-1$.
Writing
\[
j_0=\Bigl\lfloor\frac{4\pi\sqrt{N}+1-K_0}{4h}\Bigr\rfloor,
\]
we have
\[
\frac{(K_0+4jh)-1-4\pi\sqrt{N}}{4h}\in(-1,1]
\iff
j\in\{j_0,j_0+1\}.
\]
Thus, for $i=1$ we get the full integral for $j<j_0$, $0$ for $j>j_0+1$, and something in between for $j\in\{j_0,j_0+1\}$; similarly, for $i=2$ we get $0$ for $j<j_0$, the full integral for $j>j_0+1$, and something in between for $j\in\{j_0,j_0+1\}$.

We calculate that $j_0=\frac{J}{2}+O(1)$. Thus, we get half the major arc contribution from $\frac{a_i}{q_i}$, up to a relative error $\ll1/J\ll h/H\ll(H/K)^{\frac19}$.

\thispagestyle{empty}
{
\bibliographystyle{alpha}
\bibliography{reference}
}
\end{document}